\documentclass{article}

\usepackage{subfigure}

\usepackage{wrapfig}

 \usepackage[final]{neurips_2024}

\usepackage[utf8]{inputenc}
\usepackage{microtype}
\usepackage{graphicx}
\usepackage{subfigure}
\usepackage{booktabs} 
\usepackage{titlesec}
\usepackage{algorithm}
\usepackage{algorithmic}

\newcounter{mycount}
\counterwithin{algorithm}{mycount}
\refstepcounter{mycount}

\usepackage[utf8]{inputenc} 
\usepackage[T1]{fontenc}
\usepackage{hyperref}   
\usepackage{url}
\usepackage{booktabs}   
\usepackage{amsfonts}   
\usepackage{nicefrac}   
\usepackage{microtype}  
\usepackage{xcolor} 
\usepackage{tabularray}
\usepackage{amsmath}
\usepackage{amssymb}
\usepackage{mathtools}
\usepackage{amsthm}
\usepackage{float}

\usepackage[capitalize,noabbrev]{cleveref}

\theoremstyle{plain}
\newtheorem{theorem}{Theorem}[section]
\newtheorem{proposition}[theorem]{Proposition}
\newtheorem{lemma}[theorem]{Lemma}
\newtheorem{corollary}[theorem]{Corollary}
\theoremstyle{definition}
\newtheorem{definition}[theorem]{Definition}
\newtheorem{assumption}[theorem]{Assumption}
\theoremstyle{remark}
\newtheorem{remark}[theorem]{Remark}
\newtheorem{Example}[theorem]{Example}

\newtheorem{Step}{Step}

\title{Fundamental Convergence Analysis of Sharpness-Aware Minimization}

\newcommand{\set}[1]{\left\{#1\right\}}
\def\ve{\varepsilon}

\def\tilde{\widetilde}

\def\dom{{\rm dom}\,}

\def\B{\mathbb B}

\def\ox{\overline{x}}

\def\ve{\varepsilon}
\def\epsilon{\varepsilon}
\def\ox{\bar{x}}

\def\C{\mathcal{C}}

\def\dom{\mbox{\rm dom}\,}

\def\st{\stackrel}

\def \N{{\rm I\!N}}
\def \R{{\rm I\!R}}

\newcommand{\dotproduct}[1]{\left\langle#1\right\rangle}
\newcommand{\brac}[1]{\left(#1\right)}
\newcommand{\sbrac}[1]{\left[#1\right]}
\newcommand{\abs}[1]{\left|#1\right|}
\newcommand{\norm}[1]{\left\|#1\right\|}

\author{%
  Pham Duy Khanh \\
  Ho Chi Minh City University of Education\\
  \texttt{khanhpd@hcmue.edu.vn} \\
  \And
  Hoang-Chau Luong \\
  VNU-HCM University of Science \\
  \texttt{lhchau20@apcs.fitus.edu.vn} \\
  \AND
  Boris S. Mordukhovich \\
  Wayne State University\\
  \texttt{boris@math.wayne.edu} \\
  \And
  Dat Ba Tran \\
  Wayne State University\\
  \texttt{tranbadat@wayne.edu} \\
}

\begin{document}

\maketitle

\begin{abstract}
  The paper investigates the fundamental convergence properties of Sharpness-Aware Minimization (SAM), a recently proposed gradient-based optimization method \citep{foret21} that significantly improves the generalization of deep neural networks. The convergence properties including the stationarity of accumulation points, the convergence of the sequence of gradients to the origin, the sequence of function values to the optimal value, and the sequence of iterates to the optimal solution are established for the method. The universality of the provided convergence analysis based on inexact gradient descent frameworks \citet{kmt23.1} allows its extensions to efficient normalized versions of SAM such as F-SAM  \citep{li2024friendly}, VaSSO  \citep{li23vasso}, RSAM \citep{liu22}, and to the unnormalized versions of SAM such as USAM \citep{maksym22}. Numerical experiments are conducted on classification tasks using deep learning models to confirm the practical aspects of our analysis.
\end{abstract}

\section{Introduction}\label{introduction}
This paper concentrates on optimization methods for solving the standard optimization problem
\begin{align}
{\rm minimize}\quad f(x)\text{ subject to }x\in\R^n,
\end{align}
where $f:\R^n\rightarrow\R$ is a continuously differentiable ($\C^1$-smooth) function. We study the convergence behavior of the gradient-based optimization algorithm \textit{Sharpness-Aware Minimization} \citep{foret21} together with its efficient practical variants \citep{liu22,li23vasso,maksym22}. Given an initial point $x^1\in\R^n,$ the original iterative procedure of SAM is designed as follows
\begin{align}\label{SAM intro}
   \quad x^{k+1}=x^k-t\nabla f\brac{x^k+\rho\frac{\nabla f(x^k)}{\norm{\nabla f(x^k)}}}
\end{align}
for all $k\in\N$, where $t>0$ and $\rho>0$ are respectively the \textit{stepsize} (in other words, the learning rate) and \textit{perturbation radius}. The main motivation for the construction algorithm is that by making the backward step $x^k+\rho\frac{\nabla f(x^k)}{\norm{\nabla f(x^k)}}$, it avoids minimizers with large sharpness, which is usually poor for the generalization of deep neural networks as shown in \citep{keskar17}.

\subsection{Lack of convergence properties for SAM due to constant stepsize}\label{sec lack}

The consistently high efficiency of SAM has driven a recent surge of interest in the analysis of the method. The convergence analysis of SAM is now a primary focus on its theoretical understanding with several works being developed recently (e.g., \citet{ahn23d, maksym22, dai23, si23}). However, none of the aforementioned studies have addressed the fundamental convergence properties of SAM, which are outlined below where the stationary accumulation point in (2) means that every accumulation point $\bar x$ of the iterative sequence $\set{x^k}$ satisfies the condition $\nabla f(\bar x)=0$.

\begin{table}[H]
\centering
\begin{tabular}{|cl|}
\hline
(1) & $\displaystyle{\liminf_{k\rightarrow \infty}}\norm{\nabla f(x^k)}=0$ \\
(2)  & stationary accumulation point\\
   (3) & $\displaystyle{\lim_{k\rightarrow \infty}}\norm{\nabla f(x^k)}=0$\\
   (4)  &  $\displaystyle{\lim_{k\rightarrow \infty}}f(x^k)= f(\bar x)$ with $\nabla f(\bar x)=0$\\
   (5)& $\set{x^k}$ converges to some $\bar x$ with $\nabla f(\bar x)=0$\\\hline
\end{tabular}
\vspace{0.2cm}
\caption{Fundamental convergence properties of smooth optimization methods }
\label{tab:fundamental}
\end{table}

The relationship between the properties above is summarized as follows:
\begin{center}
(1)$\overset{\set{\norm{x^k}}\not\rightarrow\infty}{\Longleftarrow}$ (2) $\Longleftarrow$ (3) $\Longleftarrow$ (5) $\Longrightarrow$ (4).
\end{center}

The aforementioned convergence properties are standard and are analyzed by various smooth optimization methods including gradient descent-type methods, Newton-type methods, and their accelerated versions together with nonsmooth optimization methods under the usage of subgradients. The readers are referred to \citet{bertsekasbook,nesterovbook,polyakbook} and the references therein for those results. The following recent publications have considered various types of convergence rates for the sequences generated by SAM as outlined below:

 (i) \citet[Theorem~1]{dai23} 
 \begin{align*}
 f(x^k)-f^*\le (1-t\mu (2-Lt))^k (f(x^0)-f^*)+ \frac{tL^2\rho^2}{2\mu (2-Lt)}
 \end{align*}
 where $L$ is the Lipschitz constant of $\nabla f$, and where $\mu$ is the constant of strong convexity constant for $f$.
 
 (ii) \citet[Theorems~3.3,~3.4]{si23}
 \begin{align*}
 \frac{1}{k}\sum_{i=1}^k \norm{\nabla f(x^i)}^2&=\mathcal{O}\brac{\frac{1}{k}+\frac{1}{\sqrt{k}}} \quad\text{ and } \quad   \frac{1}{k}\sum_{i=1}^k \norm{\nabla f(x^i)}^2=\mathcal{O}\brac{\frac{1}{k}}+L^2\rho^2,
\end{align*}
 where $L$ is the Lipschitz constant of $\nabla f.$ We emphasize that none of the results mentioned above achieve any fundamental convergence properties listed in Table~\ref{tab:fundamental}. The estimation in (i) only gives us the convergence of the function value sequence to a value close to the optimal one, not the convergence to exactly the optimal value. Additionally, it is evident that the results in (ii) do not imply the convergence of $\set{\nabla f(x^k)}$ to $0$. To the best of our knowledge, the only work concerning the fundamental convergence properties listed in Table~\ref{tab:fundamental} is \citet{maksym22}. However, the method analyzed in that paper is unnormalized SAM (USAM), a variant of SAM with the norm being removed in the iterative procedure \eqref{SAM}. Recently, \citet{dai23} suggested that USAM has different effects in comparison with SAM in both practical and theoretical situations, and thus, they should be addressed separately. This observation once again highlights the necessity for a fundamental convergence analysis of SAM and its normalized variants.

Note that, using exactly the iterative procedure \eqref{SAM intro}, SAM does not achieve the convergence for either $\set{x^k}$, or $\set{f(x^k)}$, or $\set{\nabla f(x^k)}$ to  the optimal solution, the optimal value, and the origin, respectively. It is illustrated by Example~\ref{exam SAM constant not conver} below dealing with quadratic functions. This calls for the necessity of employing an alternative stepsize rule for SAM. Scrutinizing the numerical experiments conducted for SAM and its variants (e.g., \citet[Subsection~C1]{foret21}, \citet[Subsection~4.1]{li23vasso}), we can observe that in fact the constant stepsize rule is not a preferred strategy. Instead, the cosine stepsize scheduler from \citep{los16}, designed to decay to zero and then restart after each fixed cycle, emerges as a more favorable approach. This observation motivates us to analyze the method under diminishing stepsize, which is standard and employed in many optimization methods including the classical gradient descent methods together with its incremental and stochastic counterparts \citep{bertsekas00}. Diminishing step sizes also converge to zero as the number of iterations increases, a condition satisfied by the practical cosine step size scheduler in each cycle.

\subsection{Our Contributions} 

\subsubsection*{Convergence of SAM and normalized variants}

We establish fundamental convergence properties of SAM in various settings. In the convex case, we consider the perturbation radii to be variable and bounded. This analysis encompasses the practical implementation of SAM with a constant perturbation radius. The results in this case are summarized in Table~\ref{table convex case}.
\begin{table}[H]
\centering
\begin{tabular}{|l|l|l|l|}
\hline
\multicolumn{1}{|c|}{\textbf{Classes of function}} &  \multicolumn{1}{c|}{\textbf{Results}}  \\ \hline
General setting  & $\liminf\nabla f(x^k)= 0$ \\ \hline
Bounded minimizer set& stationary accumulation point \\ \hline
Unique minimizer& $\set{x^k}$ is convergent \\ \hline
\end{tabular}\vspace{0.1cm}
\caption{Convergence properties of SAM for convex functions in Theorem~\ref{theo SAM constant}}\label{table convex case}
\end{table}
\vspace{-0.7cm}
In the nonconvex case, we present a unified convergence analysis framework that can be applied to most variants of SAM, particularly recent efficient developments such as VaSSO \citep{li23vasso}, F-SAM \citep{li2024friendly}, and RSAM \citep{liu22}. We observe that all these methods can be viewed as inexact gradient descent (IGD) methods with absolute error. This version of IGD has not been previously considered, and its convergence analysis is significantly more complex than the one in \citet{kmt23.1,kmt23.2,kmpt23.3,kmt23.4}, as the function value sequence generated by the new algorithm may not be decreasing. This disrupts the convergence framework for monotone function value sequences used in the aforementioned works. To address this challenge, we adapt the analysis for algorithms with nonmonotone function value sequences from \citet{li23}, which was originally developed for random reshuffling algorithms, a context entirely different from ours.

We establish the convergence of IGD with absolute error when the perturbation radii decrease at an \textit{arbitrarily slow rate}. Although the analysis of this general framework does not theoretically cover the case of a constant perturbation radius, it poses no issues for the practical implementation of these methods, as discussed in Remark~\ref{rmk no harm}. A summary of our results in the nonconvex case is provided in the first part of Table~\ref{tab:convergence SAM and USAM}.
\subsubsection*{Convergence of USAM and unnormalized variants}
\vspace{-0.2cm}

Our last theoretical contribution in this paper involves a refined convergence analysis of USAM in \citet{maksym22}. In the general setting, we address functions satisfying the $L$-descent condition \eqref{descent condition}, which is even weaker than the Lipschitz continuity of $\nabla f$ as considered in \citet{maksym22}. The summary of the convergence analysis for USAM is given in the second part of Table~\ref{tab:convergence SAM and USAM}.

As will be discussed in Remark~\ref{rmk comparison USAM}, our convergence properties for USAM use weaker assumptions and cover a broader range of applications in comparison with those analyzed in \citep{maksym22}. Furthermore, the universality of the conducted analysis allows us to verify all the convergence properties for the extragradient method \citep{korpelevich76} that has been recently applied in \citep{lin20} to large-batch training in deep learning.
\vspace{-0.2cm}

\begin{table}[H]
\centering
\begin{tabular}{|l|l|l|l|} 
\hline
\multicolumn{2}{|c|}{\textbf{SAM and normalized variants}}& \multicolumn{2}{|c|}{\textbf{USAM and unnormalized variants}}\\ 
\hline
\multicolumn{1}{|c|}{\textbf{Classes of functions}} & \multicolumn{1}{|c|}{\textbf{Results}} & \multicolumn{1}{|c|}{\textbf{Classes of functions}} & \multicolumn{1}{|c|}{\textbf{Results}} \\ 
\hline
General setting & $\lim\nabla f(x^k)= 0$   & General setting & stationary accumulation point \\\hline  
General setting & $\lim f(x^k)= f^*$  & General setting & $\lim f(x^k)= f^*$ \\\hline
KL property & $\set{x^k}$ is convergent & $\nabla f$ is Lipschitz& $\lim \nabla f(x^k)= 0$ \\\hline
&& KL property & $\set{x^k}$ is convergent \\ 
\hline
\end{tabular}\vspace{0.1cm}
\caption{Convergence properties of SAM together with normalized variants (Corollary~\ref{coro general}, Appendix~\ref{versions}), and USAM together with unnormalized variants (Theorem~\ref{theorem USAM EG})}\label{tab:convergence SAM and USAM}
\end{table}

\subsection{Importance of Our Work}

Our work develops, for the first time in the literature, a fairly comprehensive analysis of the fundamental convergence properties of SAM and its variants. The developed approach addresses general frameworks while being based on the analysis of the newly proposed inexact gradient descent methods. Such an approach can be applied in various other circumstances and provides useful insights into the convergence understanding of not only SAM and related methods but also many other numerical methods in smooth, nonsmooth, and derivative-free optimization.

\subsection{Related Works}

\textbf{Variants of SAM}. There have been several publications considering some variants to improve the performance of SAM. Namely, \citet{kwon21} developed the Adaptive Sharpness-Aware Minimization (ASAM) method by employing the concept of normalization operator. \citet{du22} proposed the Efficient Sharpness-Aware Minimization (ESAM) algorithm by combining stochastic weight perturbation and sharpness-sensitive data selection techniques. \citet{liu22} proposed a novel Random Smoothing-Based SAM method called RSAM that improves the approximation quality in the backward step. Quite recently, \citet{li23vasso} proposed another approach called Variance Suppressed Sharpness-aware Optimization (VaSSO), which perturbed the backward step by incorporating information from the previous iterations. As \citet{li2024friendly} identified noise in stochastic gradient as a crucial factor in enhancing SAM's performance, they proposed Friendly Sharpness-Aware Minimization (F-SAM) which perturbed the backward step by extracting noise from the difference between the stochastic gradient and the expected gradient at the current step. Two efficient algorithms, AE-SAM and AE-LookSAM, are also proposed in \citet{jiang23}, by employing adaptive policy based on the loss landscape geometry.

\textbf{Theoretical Understanding of SAM}. Despite the success of SAM in practice, a theoretical understanding of SAM was absent until several recent works. \citet{barlett22} analyzed the convergence behavior of SAM for convex quadratic objectives, showing that for most random initialization, it converges to a cycle that oscillates between either side of the minimum in the direction with the largest curvature. \citet{ahn23d} introduces the notion of $\varepsilon$-approximate flat minima and investigates the iteration complexity of optimization methods to find such approximate flat minima. As discussed in Subsection~\ref{sec lack}, \citet{dai23} considers the convergence of SAM with constant stepsize and constant perturbation radius for convex and strongly convex functions, showing that the sequence of iterates stays in a neighborhood of the global minimizer while \citep{si23} considered the properties of the gradient sequence generated by SAM in different settings.

\noindent\textbf{Theoretical Understanding of USAM}. This method was first proposed by \citet{maksym22} with fundamental convergence properties being analyzed under different settings of convex and nonconvex and optimization. Analysis of USAM was further conducted in \citet{behdin16} for a linear regression model, and in \citet{argawal23} for a quadratic regression model. Detailed comparison between SAM and USAM, which indicates that they exhibit different behaviors, was presented in the two recent studies by \citet{compa23} and \citet{dai23}. During the final preparation of the paper, we observed that the convergence of USAM can also be deduced from \citet{manga97}, though under some additional assumptions, including the boundedness of the gradient sequence.

\section{Preliminaries} \label{sec prelim}

First we recall some notions and notations frequently used in the paper. All our considerations are given in the space $\R^n$ with the Euclidean norm $\|\cdot\|$. 
As always, $\N:=\{1,2,\ldots\}$ signifies the collections of natural numbers. The symbol $x^k\st{J}{\to}\ox$ means that $x^k\to\ox$ as $k\to\infty$ with $k\in J\subset\N$. Recall that $\bar x$ is a \textit{stationary point} of a $\mathcal{C}^1$-smooth function $f\colon\R^n\rightarrow\R$ if $\nabla f(\bar x)=0$. A function $f:\R^n\rightarrow {\R}$ is said to posses a Lipschitz continuous gradient with the uniform constant $L>0$, or equivalently it belongs to the class $\C^{1,L}$, if we have the estimate
\begin{align}\label{Lips def}
\norm{\nabla f(x)-\nabla f(y)}\le L\norm{x-y}\;\text{ for all }\;x,y\in\R^n.
\end{align}
This class of function enjoys the following property called the {\em $L$-descent condition} (see, e.g., \citet[Lemma~A.11]{solodovbook} and \citet[Lemma~A.24]{bertsekasbook}):
\begin{align}\label{descent condition}
f(y)\le f(x)+\dotproduct{\nabla f(x),y-x}+\dfrac{L}{2}\norm{y-x}^2
\end{align}
for all $x,y\in\R^n.$
Conditions \eqref{Lips def} and \eqref{descent condition} are equivalent to each other when $f$ is convex, while the equivalence fails otherwise. A major class of functions satisfying the $L$-descent condition but not having the Lipschitz continuous gradient is given by \citet[Section~2]{kmt23.1} as $f(x):=\dfrac{1}{2}\dotproduct{Ax,x}+\dotproduct{b,x}+c-h(x),$
where $A$ is an $n\times n$ matrix, $b\in\R^n$, $c\in\R$ and $h:\R^n\rightarrow \R$ is a smooth convex function whose gradient is not Lipschitz continuous. There are also circumstances where a function has a Lipschitz continuous gradient and satisfies the descent condition at the same time, but the  Lipschitz constant is larger than the one in the descent condition.

Our convergence analysis of the methods presented in the subsequent sections benefits from the {\em Kurdyka-\L ojasiewicz $(KL)$ property} taken from \citet{attouch10}.

\begin{definition} [Kurdyka-\L ojasiewicz property] \rm \label{KL ine}\rm We say that a smooth function $f:\R^n\rightarrow {\R}$ enjoys the \textit{KL property} at $\bar x\in\dom \partial f$ if there exist $\eta\in (0,\infty]$, a neighborhood $U$ of $\bar x$, and a desingularizing concave continuous function $\varphi:[0,\eta)\rightarrow[0,\infty)$ such that $\varphi(0)=0$, $\varphi$ is $\mathcal{C}^1$-smooth on $(0,\eta)$, $\varphi'>0$ on $(0,\eta)$, and for all $x\in U$ with $0<f(x)-f(\bar x)<\eta$, we have 
\begin{align}\label{KL 2}
\varphi'(f(x)-f(\bar x))\norm{\nabla f(x)}\ge 1.
\end{align}
\end{definition}

\begin{remark}\rm\label{algebraic} It has been realized that the KL property is satisfied in broad settings. In particular, it holds at every {\em nonstationary point} of $f$; see \citet[Lemma~2.1~and~Remark~3.2(b)]{attouch10}. Furthermore, it is proved at the seminal paper \citet{lojasiewicz65} that any analytic function $f:\R^n\rightarrow\R$ satisfies the KL property on $\R^n$ with $\varphi(t)~=~Mt^{1-q}$ for some $q\in [0,1)$. Typical functions that satisfy the KL property are {\em semi-algebraic} functions and in general, functions {\em definable in o-minimal structures}; see \citet{attouch10,attouch13,kurdyka}.
\end{remark}
We utilize the following assumption on the desingularizing function in Definition~\ref{KL ine}, which is employed in \citep{li23}. The satisfaction of this assumption for a general class of desingularizing functions is discussed in Remark~\ref{rmk satisfaction}.
\begin{assumption}\label{assu desi}
There is some $C>0$ such that whenever $x,y\in (0,\eta)$ with $x+y<\eta$, it holds that
$$C[\varphi'(x+y)]^{-1}\le (\varphi'(x))^{-1}+(\varphi'(y))^{-1}.$$ 
\end{assumption}

\section{SAM and normalized variants}\label{sec SAM normal}

\subsection{Convex case}
We begin this subsection with an example that illustrates SAM's inability to achieve the convergence of the sequence of iterates to an optimal solution of strongly convex quadratic functions by using a constant stepsize. This emphasizes the necessity of avoiding this type of stepsize in our subsequent analysis.
\begin{Example}[SAM with constant stepsize and constant perturbation radius does not converge]\rm  \label{exam SAM constant not conver} Let the sequence $\set{x^k}$ be generated by SAM in \eqref{SAM intro} applied to the strongly convex quadratic function $f(x)=\frac{1}{2}\dotproduct{Ax,x}-\dotproduct{b,x}$, where $A$ is an $n\times n$ symmetric, positive-definite matrix and $b\in\R^n.$ Then for any fixed small constant perturbation radius and for some small constant stepsize together with an initial point close to the solution, the sequence $\set{x^k}$ generated by this algorithm does not converge to the optimal solution. 
\end{Example}
\begin{figure}
    \centering
    \includegraphics[scale=0.3]{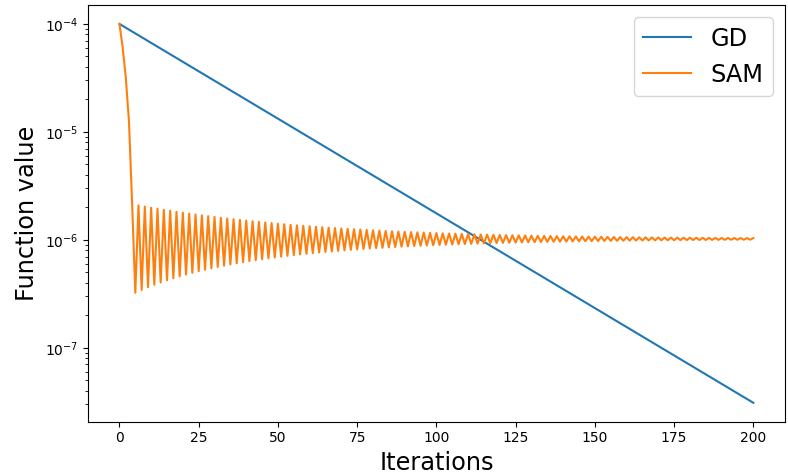}
    \caption{SAM with constant stepsize does not converge to optimal solution}
    \label{fig:GD vs SAM}
\end{figure}
The details of the above example are presented in Appendix~\ref{appendix exam SAM constant not conver}. Figure~\ref{fig:GD vs SAM} gives an empirical illustration for Example~\ref{exam SAM constant not conver}. The graph shows that, while the sequence generated by GD converges to $0$, the one generated by SAM gets stuck at a different point.

As the constant stepsize does not guarantee the convergence of SAM, we consider another well-known stepsize called diminishing (see \eqref{parameter constant rho}). The following result provides the convergence properties of SAM in the convex case for that type of stepsize.
\begin{theorem}\label{theo SAM constant} Let $f:\R^n\rightarrow\R$ be a smooth convex function whose gradient is Lipschitz continuous with constant $L>0$. Given any initial point $x^1\in\R^n$, let $\set{x^k}$ be generated by the SAM method with the iterative procedure
\begin{align}\label{iterative SAM}
 x^{k+1}=x^k-t_k\nabla f\brac{x^k+\rho_k\frac{\nabla f(x^k)}{\norm{\nabla f(x^k)}}}
\end{align}
for all $k\in \N$ with nonnegative stepsizes and  perturbation radii satisfying the conditions
\begin{align}\label{parameter constant rho}
\sum_{k=1}^\infty t_k^2<\infty,\; \sum_{k=1}^\infty t_k=\infty,\; \sup_{k\in\N}\rho_k<\infty.
\end{align}
Assume that $\nabla f(x^k)\ne 0$ for all $k\in\N$ and that $\inf_{k\in\N} f(x^k)>-\infty$. Then the following assertions hold:

{\bf(i)} $\displaystyle{\liminf_{k\rightarrow \infty}}\;\nabla f(x^k)=0$.
 
{\bf(ii)} If $f$ has a nonempty bounded level set, then $\set{x^k}$ is bounded, every accumulation point of $\set{x^k}$ is a global minimizer of $f$, and $\set{f(x^k)}$ converges to the optimal value of $f$. If in addition $f$ has a unique minimizer, then the sequence $\set{x^k}$ converges to that minimizer.
\end{theorem}
Due to the space limit, the proof of the theorem is presented in Appendix~\ref{appendix theo SAM constant}.

\subsection{Nonconvex case}

In this subsection, we study the convergence of several versions of SAM from the perspective of the inexact gradient descent methods.

\setcounter{algorithm}{-1}
\begin{algorithm}[H]
\caption{Inexact Gradient Descent (IGD) Methods }\label{IGD}
\setcounter{Step}{-1}
\begin{Step}\rm
Choose some initial point $x^0\in\R^n,$ sequence of errors $\set{\varepsilon_k}\subset (0,\infty)$, and sequence of stepsizes $\set{t_k}\subset(0,\infty).$ For $k=1,2,\ldots,$ do the following
\end{Step}
\begin{Step}\rm 
Set $x^{k+1}=x^k-t_k g^k$ with $\norm{g^k-\nabla f(x^k)}\le \varepsilon_k$.
\end{Step} 
\end{algorithm}
This algorithm is motivated by while being different from the Inexact Gradient Descent methods proposed in \citep{kmpt23.3,kmt23.1, kmt23.4,kmt23.2}. The latter constructions consider relative errors in gradient calculation, while Algorithm~\ref{IGD} uses the absolute ones. This approach is particularly suitable for the constructions of SAM and its normalized variants. The convergence properties of Algorithm~\ref{IGD} are presented in the next theorem.
\begin{theorem}\label{convergence IGD}
Let $f:\R^n\rightarrow\R$ be a smooth function whose gradient is Lipschitz continuous with some constant $L>0$, and let $\set{x^k}$ be generated by the IGD method in Algorithm~\ref{IGD} with stepsizes and errors satisfying the conditions
   \begin{align}\label{parameter general}
\sum_{k=1}^\infty t_k=\infty,\;t_k\downarrow0, \sum_{k=1}^\infty t_k \varepsilon_k<\infty,\;\limsup \varepsilon_k<2.
\end{align}
Assume that $\inf_{k\in\N} f(x^k)>-\infty$. Then the following convergence properties hold:

{\bf(i)} $\nabla f(x^k)\rightarrow0,$ and thus every accumulation point of $\set{x^k}$ is stationary for $f.$

{\bf(ii)} If $\bar x$ is an accumulation point of the sequence $\set{x^k}$, then  $f(x^k)\rightarrow f(\bar x).$ 

{\bf(iii)} Suppose that $f$ satisfies the KL property  at some accumulation point $\bar x$ of $\set{x^k}$ with the desingularizing function $\varphi$ satisfying Assumption~\ref{assu desi}. Assume in addition that 
\begin{align}\label{desing condi}
\sum_{k=1}^\infty t_k\brac{\varphi'\brac{\sum_{i=k}^\infty t_k\varepsilon_k}}^{-1}<\infty,
\end{align} and that $f(x^k) > f(\bar x)$ for sufficiently large $k \in \mathbb{N}$. Then $x^k \rightarrow \bar x$ as $k\rightarrow\infty$. In particular, if $\bar x$ is a global minimizer of $f$, then either $f(x^k) = f(\bar x)$ for some $k \in \mathbb{N}$, or $x^k \rightarrow \bar x$. 
\end{theorem}
The proof of the theorem is presented in Appendix~\ref{appendix convergence IGD}. The demonstration that condition \eqref{desing condi} is satisfied when $\varphi(t) = Mt^{1-q}$ with some $M > 0$ and $q \in (0,1)$, and when $t_k = \frac{1}{k}$ and $\varepsilon_k = \frac{1}{k^p}$ with $p \ge 0$ for all $k \in \mathbb{N}$, is provided in Remark~\ref{rmk desing condi}.

The next example discusses the necessity of the last two conditions in \eqref{parameter general} in the convergence analysis of IGD while demonstrating that employing a constant error leads to the convergence to a nonstationary point of the method.

\begin{Example}[IGD with constant error converges to a nonstationary point]\rm \label{exam IGD not conve} Let $f:\R\rightarrow\R$ be defined by $f(x)=x^2$ for $x\in\R$. Given a perturbation radius $\rho>0$ and an initial point $x^1>\rho$, consider the iterative sequence
\begin{align}\label{IGD donot}
x^{k+1}=x^k-2t_k\brac{x^k-\rho \frac{x^k}{|x^k|}}\;\text{ for  }k\in\N,
\end{align}
where $\set{t_k}\subset[0,1/2],t_k\downarrow0,$ and $\sum_{k=1}^\infty t_k=\infty$. This algorithm is in fact the IGD applied to $f$ with $g^k= \nabla f\brac{x^k-\rho\frac{f'(x^k)}{\abs{f'(x^k)}}}$. Then  $\set{x^k}$ converges to $\rho,$ which is not a stationary point of $f.$ 
\end{Example}
The details of the example are presented in Appendix~\ref{appendix exam IGD not}. We now propose a general framework that encompasses SAM and all of its normalized variants including RSAM \citep{liu22}, VaSSO  \citep{li23vasso} and F-SAM \citep{li2024friendly}. Due to the page limit, we refer readers to Appendix~\ref{versions} for the detailed constructions of those methods. Remark~\ref{normalized variants remark} in Appendix~\ref{versions} also shows that all of these methods are special cases of Algorithm~\ref{general}, and thus all the convergence properties presented in Theorem~\ref{convergence IGD} follow. 
\begin{algorithm}[H]\caption{General framework for normalized variants of SAM}\label{general}
\setcounter{Step}{-1}
 \begin{Step}\rm
Choose  $x^1\in\R^n,$  $\set{\rho_k},\set{t_k}\subset (0,\infty)$, and $\set{d^k}\subset\R^n\setminus\set{0}.$ For $k\ge 1$ do the following:
\end{Step}
\begin{Step}\rm 
Set $x^{k+1}=x^k-t_k \nabla f\brac{x^k+\rho_k \frac{d^k}{\norm{d^k}}}$.
\end{Step}   

\end{algorithm}

\begin{corollary}\label{coro general} Let $f:\R^n\rightarrow\R$ be a $\C^{1,L}$ function, and let $\set{x^k}$ be generated by Algorithm~\ref{general} with the parameters
\begin{align}\label{parameter SAM + VASSO}
 \sum_{k=1}^\infty t_k=\infty,\;t_k\downarrow0, \sum_{k=1}^\infty t_k \rho_k<\infty,\;\limsup \rho_k<\frac{2}{L}.
\end{align}
Assume that $\inf_{k\in\N} f(x^k)>-\infty$. Then all convergence properties presented in Theorem~\ref{convergence IGD} hold.
\end{corollary}
The proof of this result is presented in Appendix~\ref{proof coro general}.
\begin{remark}\rm \label{rmk no harm}
  Note that the conditions in \eqref{parameter SAM + VASSO} do not pose any obstacles to the implementation of a constant perturbation radius for SAM in practical circumstances. This is due to the fact that a possible selection of ${t_k}$ and ${\rho_k}$ satisfying \eqref{parameter SAM + VASSO} is $t_k = \frac{1}{k}$ and $\rho_k = \frac{C}{k^{0.001}}$ for all $k\in\mathbb{N}$ (almost constant), where $C>0$. Then the initial perturbation radius is $C$, while after $C$ million iterations, it remains greater than $0.99C$. This phenomenon is also confirmed by numerical experiments in Appendix~\ref{numerical support} on nonconvex functions. The numerical results show that SAM with almost constant radii $\rho_k = \frac{C}{k^{p}}$ has a similar convergence behavior to SAM with a constant radius $\rho=C$. As SAM with a constant perturbation radius has sufficient empirical evidence for its efficiency in \citet{foret21}, this also supports the practicality of our almost constant perturbation radii.
\end{remark}

\section{USAM and unnormalized variants} \label{sec uSAM}

In this section, we study the convergence of various versions of USAM from the perspective of the following Inexact Gradient Descent method with relative errors. 
\setcounter{mycount}{2}
\setcounter{algorithm}{-1}
 \begin{algorithm}[H]
\caption{IGDr} \label{IGDr}
\setcounter{Step}{-1}
 \begin{Step}\rm
Choose some $x^0\in\R^n,\nu\ge 0$, and  $\set{t_k}\subset[0,\infty).$ For $k=1,2,\ldots,$ do the following:
\end{Step}
\begin{Step}\rm 
Set $x^{k+1}=x^k-t_k g^k$, where $\norm{g^k-\nabla f(x^k)}\le \nu \norm{\nabla f(x^k)}$.
\end{Step}   
\end{algorithm}

This algorithm was initially introduced in \citet{kmt23.1} in a different form, considering a different selection of error. The form of IGDr closest to Algorithm~\ref{IGDr} was established in \citet{kmt23.2} and  then further studied in  \citet{kmt23.2,kmpt23.3,kmt23.4}. In this paper, we extend the analysis of the method to a general stepsize rule covering both constant and diminishing cases, which was not considered in \citet{kmt23.2}.

\begin{theorem}\label{theo IGDr}
Let $f:\R^n\rightarrow\R$ be a smooth function satisfying the descent condition for some constant $L>0,$ and let $\set{x^k}$ be the sequence generated by Algorithm~\ref{IGDr} with  the relative error $\nu \in [0,1)$, and the stepsizes satisfying 
\begin{align}\label{stepsize IGDr}
   \sum_{k=1}^\infty t_k=\infty\;\text{ and }\; t_k\in \sbrac{0,\frac{2-2\nu-\delta}{L(1+\nu)^2}}
\end{align}
 for sufficiently large $k\in \N$ and for some $\delta>0$. Then either $f(x^k)\rightarrow-\infty$, or we have the assertions:
   
{\bf(i)} Every accumulation point of $\set{x^k}$ is a stationary point of the cost function $f$.

{\bf(ii)} If the sequence $\set{x^k}$ has any accumulation point $\bar x$, then $f(x^k)\downarrow f(\bar x)$.

{\bf(iii)} If $f\in\C^{1,L}$, then $\nabla f(x^k)\rightarrow0.$
   
   {\bf(iv)} If $f$ satisfies the KL property at some accumulation point $\bar x$ of $f$, then $\set{x^k}\rightarrow\bar x$.

{\bf(v)} Assume in addition to (iv) that the stepsizes are bounded away from $0$, and the KL property in {\rm(iv)} holds with the desingularizing function $\varphi(t)=Mt^{1-q}$ with $M>0$ and $q\in (0,1)$. Then either $\set{x^k}$ stops finitely at a stationary point, or the following convergence rates are achieved:
 
{$\bullet$} If $q=1/2$, then $\set{x^k},\;\set{\nabla f(x^k)}$, $\set{f(x^k)}$ converge linearly as $k\to\infty$ to $\bar x$, $0$, and $f(\bar x)$.

{$\bullet$} If $q\in(1/2,1)$, then 
\begin{align*}
\norm{x^k-\bar x}=\mathcal{O}\brac{k^{-\frac{1-q}{2q-1}}}&,\;\norm{\nabla f(x^k)}=\mathcal{O}\brac{k^{-\frac{1-q}{2q-1}}},\;f(x^k)-f(\bar x)=\mathcal{O}\brac{k^{-\frac{2-2q}{2q-1}}}.
\end{align*}
\end{theorem}
Although the ideas for proving this result is similar to the one given in \citet{kmt23.2}, we do provide the full proof in the Appendix~\ref{appendix theo IGDr} for the convenience of the readers. We now show that using this approach, we derive more complete convergence results for USAM in \citet{maksym22} and also the extragradient method by \citet{korpelevich76,lin20}.

 \begin{algorithm}[H]
\caption{\citep{maksym22} Unnormalized Sharpness-Aware Minimization (USAM)} \label{USAM}
\setcounter{Step}{-1}
 \begin{Step}\rm
Choose $x^0\in\R^n,$ $\set{\rho_k}\subset [0,\infty)$, and $\set{t_k}\subset[0,\infty).$ For $k=1,2,\ldots,$ do the following:
\end{Step}
\begin{Step}\rm 
Set $x^{k+1}=x^k-t_k \nabla f(x^k+\rho_k \nabla f(x^k))$.
\end{Step}   
\end{algorithm}

 \begin{algorithm}[H]
\caption{\citep{korpelevich76} Extragradient Method} \label{EG}
\setcounter{Step}{-1}
 \begin{Step}\rm
Choose $x^0\in\R^n,$ $\set{\rho_k}\subset [0,\infty)$, and $\set{t_k}\subset[0,\infty).$ For $k=1,2,\ldots,$ do the following:
\end{Step}
\begin{Step}\rm 
Set $x^{k+1}=x^k-t_k \nabla f(x^k-\rho_k \nabla f(x^k))$.
\end{Step}   
\end{algorithm}

We are ready now to derive convergence of the two algorithms above. The proof of the theorem is given in Appendix~\ref{proof USAM EG}
\begin{theorem}\label{theorem USAM EG}
Let $f:\R^n\rightarrow\R$ be a ${\cal C}^1$-smooth function satisfying the descent condition with some constant $L>0.$ Let $\set{x^k}$ be the sequence generated by either Algorithm~\ref{USAM}, or by Algorithm~\ref{EG} with $\rho_k\le \frac{\nu}{L}$ for some $\nu \in[0,1)$ and with the stepsize satisfying \eqref{stepsize IGDr}. Then all the convergence properties in Theorem~\ref{theo IGDr} hold. 
\end{theorem}
\section{Numerical Experiments}\label{sec numerical}
To validate the practical aspect of our theory, this section compares the performance of SAM employing constant and diminishing stepsizes in image classification tasks. All the experiments are conducted on a computer with NVIDIA RTX 3090 GPU. The three types of diminishing stepsizes considered in the numerical experiments are $\eta_1/n$ (Diminish 1), $\eta_1/n^{0.5001}$ (Diminish 2), and $\eta_1/m\log m$ (Diminish 3), where $\eta_1$ is the initial stepsize, $n$ represents the number of epochs performed, and $m = \lfloor n/5 \rfloor + 2$. The constant stepsize in SAM is selected through a grid search over {0.1, 0.01, 0.001} to ensure a fair comparison with the diminishing ones. The algorithms are tested on two widely used image datasets: CIFAR-10 \citep{krizhevsky2009learning} and CIFAR-100 \citep{krizhevsky2009learning}.

\textbf{CIFAR-10}. We train well-known deep neural networks including ResNet18 \citep{he2016deep}, ResNet34 \citep{he2016deep}, and WideResNet28-10 \citep{zagoruyko2016wide} on this dataset by using 10\% of the training set as a validation set. Basic transformations, including random crop, random horizontal flip, normalization, and cutout \citep{devries2017improved}, are employed for data augmentation. All the models are trained by using SAM with SGD Momentum as the base optimizer for $200$ epochs and a batch size of $128$. This base optimizer is also used in the original paper \citep{foret21} and in the recent works on SAM \citet{ahn23d,li23vasso}. Following the approach by \citet{foret21}, we set the initial stepsize to $0.1$, momentum to $0.9$, the $\ell_2$-regularization parameter to $0.001$, and the perturbation radius $\rho$ to $0.05$. Setting the perturbation radius to be a constant here does not go against our theory, since by Remark~\ref{rmk no harm}, SAM with a constant radius and our almost constant radius have the same numerical behavior. We also conducted the numerical experiment with an almost constant radius and got the same results. Therefore, for simplicity of presentation, a constant perturbation radius is chosen. The algorithm with the highest accuracy, corresponding to the best performance, is highlighted in bold. The results in Table \ref{table: experiment momentum} report the mean and 95\% confidence interval across the three independent runs. The training loss in several tests is presented in Figure \ref{fig: experiment momentum}.

\textbf{CIFAR-100}. The training configurations for this dataset are similar to CIFAR10. The accuracy results are presented in Table \ref{table: experiment momentum}, while the training loss results are illustrated in Figure \ref{fig: experiment momentum}.

\begin{figure}[H]
\centering
\includegraphics[width=.23\textwidth]{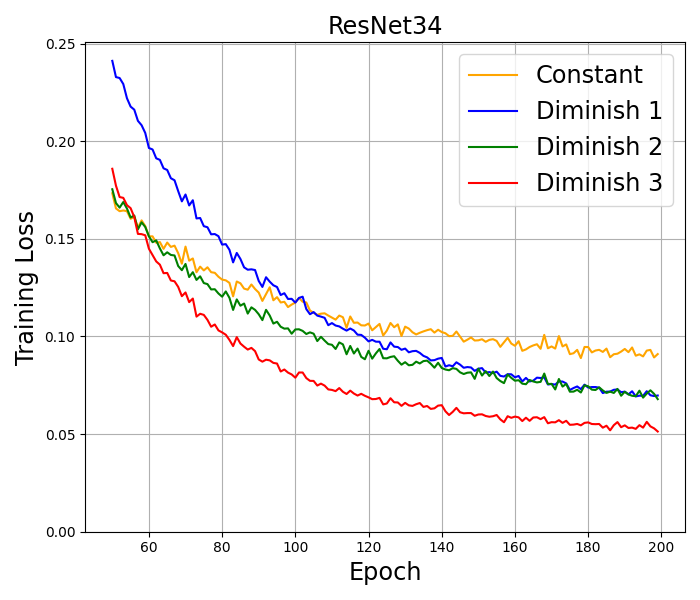}\quad
\includegraphics[width=.23\textwidth]{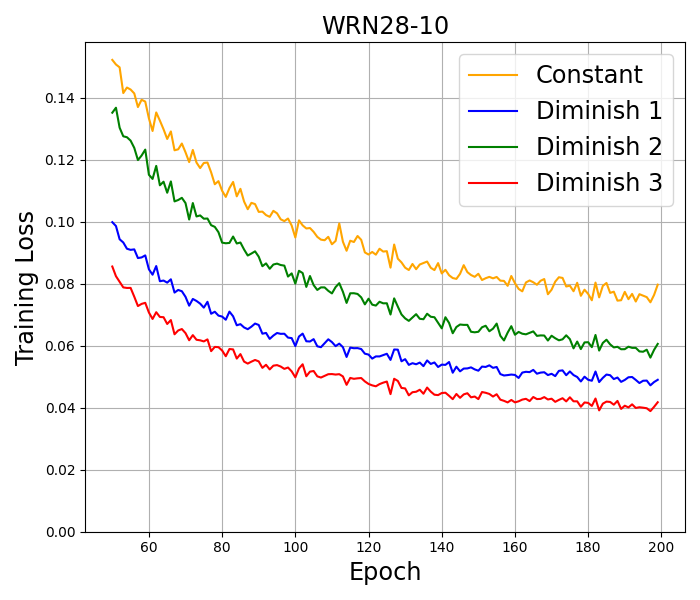}\quad
\includegraphics[width=.23\textwidth]{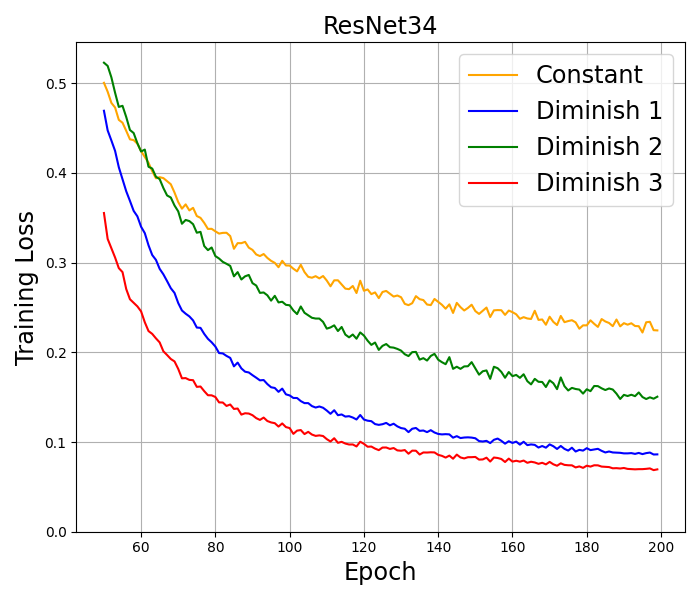}\quad
\includegraphics[width=.23\textwidth]{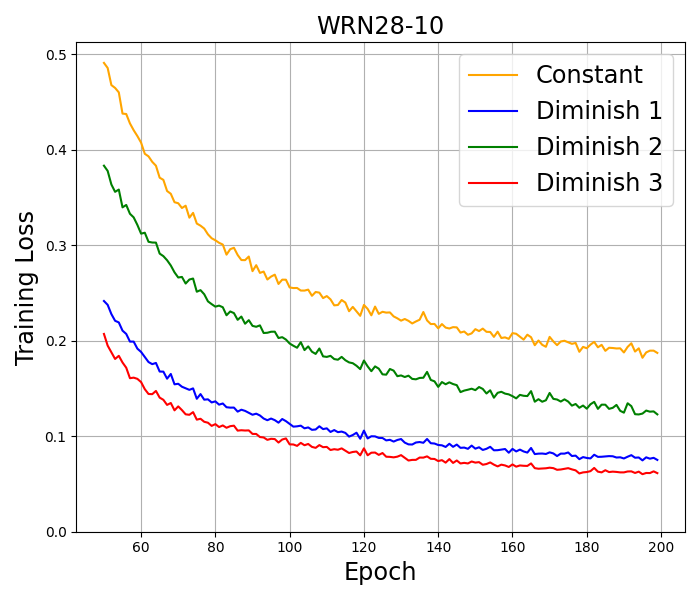}
 \caption{Training loss on CIFAR-10 (first two graphs) and CIFAR-100 (last two graphs)}
\label{fig: experiment momentum}
\end{figure}\vspace{-0.4cm}

\begin{table}[htbp]\small
\centering
\begin{tblr}{
  column{even} = {r},
  column{3} = {r},
  column{5} = {r},
  column{7} = {r},
  cell{1}{2} = {c=3}{c},
  cell{1}{5} = {c=3}{c},
  hline{1-3,7} = {-}{},
}
 & CIFAR-10& & & CIFAR-100   & & \\
Model & ResNet18& ResNet34& WRN28-10& ResNet18& ResNet34& WRN28-10\\
Constant & 94.10  	$\pm$ 0.27& 94.38  	$\pm$ 0.47& 95.33  	$\pm$ 0.23& 71.77  	$\pm$ 0.26& 72.49  	$\pm$ 0.23& 74.63  	$\pm$ 0.84\\
Diminish 1   & 93.95  	$\pm$ 0.34& 93.94  	$\pm$ 0.40& 95.18  	$\pm$ 0.03& 74.43  	$\pm$ 0.12& 73.99  	$\pm$ 0.70& 78.59  	$\pm$ 0.03\\
Diminish 2   & 94.60  	$\pm$ 0.09& \textbf{95.09  	$\pm$ 0.16} & 95.75  	$\pm$ 0.23& 73.40  	$\pm$ 0.24& 74.44  	$\pm$ 0.89& 77.04  	$\pm$ 0.23\\
Diminish 3   & \textbf{94.75  	$\pm$ 0.20} & 94.47  	$\pm$ 0.08& \textbf{95.88  	$\pm$ 0.10} & \textbf{75.65  	$\pm$ 0.44} & \textbf{74.92  	$\pm$ 0.76} & \textbf{79.70  	$\pm$ 0.12} 
\label{table: experiment momentum}
\end{tblr}\vspace{0.1cm}
\caption{Test accuracy on CIFAR-10 and CIFAR-100}
\end{table}\vspace{-0.4cm}
The results on CIFAR-10 and CIFAR-100 indicate that SAM with \textbf{Diminish 3} stepsize usually achieves the best performance in both accuracy and training loss among all tested stepsizes. In all the architectures used in the experiment, the results consistently show that diminishing stepsizes outperform constant stepsizes in terms of both accuracy and training loss measures. Additional numerical results on a larger data set and without momentum can be found in Appendix~\ref{nume addi}.
\vspace{-0.3cm}

\section{Discusison}\label{sec discuss}
\subsection{Conclusion}

In this paper, we provide a fundamental convergence analysis of SAM and its normalized variants together with a refined convergence analysis of USAM and its unnormalized variants. Our analysis is conducted in deterministic settings under standard assumptions that cover a broad range of applications of the methods in both convex and nonconvex optimization. The conducted analysis is universal and thus can be applied in different contexts other than SAM and its variants. The performed numerical experiments show that our analysis matches the efficient implementations of SAM and its variants that are used in practice.
\vspace{-0.3cm}

\subsection{Limitations}\vspace{-0.2cm}

 Our analysis is only conducted in deterministic settings, which leaves the stochastic and random reshuffling developments to our future research. The analysis of SAM coupling with momentum methods is also not considered in this paper. Another limitation pertains to numerical experiments, where only SAM was tested on three different architectures of deep learning.

\section*{Acknowledgment}

Pham Duy Khanh, Research of this author is funded by the Ministry of Education and Training Research Funding under the project B2024-SPS-07. Boris S. Mordukhovich, Research of this author was partly supported by the US National Science Foundation under grants DMS-1808978 and DMS-2204519, by the Australian Research Council under grant DP-190100555, and by Project 111 of China under grant D21024. Dat Ba Tran, Research of this author was partly supported by the US National Science Foundation under grants DMS-1808978 and DMS-2204519. 

The authors would like to thank Professor Mikhail V. Solodov for his fruitful discussions on the convergence of variants of SAM.

\bibliographystyle{plainnat} 
\bibliography{reference}

\newpage
\appendix
\tableofcontents
\onecolumn

\section{Counterexamples illustrating the Insufficiency of Fundamental Convergence Properties}
\subsection{Proof of Example~\ref{exam SAM constant not conver}}\label{appendix exam SAM constant not conver}
\begin{proof}
{Since $f(x)=\frac{1}{2}\dotproduct{Ax,x}-\dotproduct{b,x}$, the gradient of $f$ and the optimal solution are given by
\begin{align*}
\nabla f(x)=Ax-b\quad\text{ and }\quad x^*=A^{-1}b.
\end{align*}
Let $\lambda_{\min},\lambda_{\max}>0$ be the minimum, maximum eigenvalues of $A$, respectively and assume that
\begin{align}\label{select}
t\in \brac{\frac{1}{\lambda_{\min}}-\frac{1}{\lambda_{\max}+\lambda_{\min}},\frac{1}{\lambda_{\min}}},\;\rho>0,\;\text{ and }\;0<\norm{x^1-x^*}<\frac{t \rho \lambda_{\min}}{1-t\lambda_{\min}}.
\end{align}
The iterative procedure of \eqref{iterative SAM} can be written as
\begin{align}\label{update univariate SAM}
x^{k+1}=x^k-t\nabla f\brac{x^k+\rho\frac{\nabla f(x^k)}{\norm{\nabla f(x^k)}}} =x^k-t\sbrac{A\brac{x^k+\rho \frac{Ax^k-b}{\norm{Ax^k-b}}}-b}.
\end{align}
 Then $\set{x^k}$ satisfies the inequalities
\begin{align}\label{induction}
0<\norm{x^k-A^{-1}b}<\frac{t \rho \lambda_{\min}}{1-t\lambda_{\min}}\;\text{ for all }\;k\in\N.
\end{align}
 It is obvious that \eqref{induction} holds for $k=1.$ Assuming that this condition holds for any $k\in\N,$ let us show that it holds for $k+1.$ We deduce from the iterative update \eqref{update univariate SAM} that
\begin{align}\label{trho to 0}
\norm{x^{k+1}-A^{-1}b}&=\norm{x^k-t\sbrac{A\brac{x^k+\rho \frac{Ax^k-b}{\norm{Ax^k-b}}}-b}-A^{-1}b}\nonumber\\
&=\norm{(I-tA)(x^k-A^{-1}b)-t\rho \frac{A(Ax^k-b)}{\norm{Ax^k-b)}}}\nonumber\\
&\ge \norm{t\rho \frac{A(Ax^k-b)}{\norm{Ax^k-b}}}-\norm{(I-tA)(x^k-A^{-1}b)}\nonumber\\
&\ge t\rho \lambda_{\min} -(1-t\lambda_{\min})\norm{x^k-A^{-1}b}>0.
\end{align}
In addition, we get
\begin{align*}
\norm{x^{k+1}-A^{-1}b}&\le \norm{t\rho \frac{A(Ax^k-b)}{\norm{Ax^k-b}}}+\norm{(I-tA)(x^k-A^{-1}b)}\\
&\le t\rho \lambda_{\max}+(1-t\lambda_{\min})\norm{x^k-A^{-1}b}\\
&\le t\rho \lambda_{\max}+t\rho \lambda_{\min}< t\rho\frac{\lambda_{\min}}{1-t\lambda_{\min}}
\end{align*}
where the last inequality follows from $t >\frac{1}{\lambda_{\min}}-\frac{1}{\lambda_{\max}+\lambda_{\min}}$ from \eqref{select}. Thus, \eqref{induction} is verified. It follows from \eqref{trho to 0} that $x^k\not\rightarrow x^*.$}
\end{proof}

\subsection{Proof of Example~\ref{exam IGD not conve}}\label{appendix exam IGD not}
\begin{proof} Observe that $x^k>\rho$ for all $k\in\N.$ Indeed, this follows from $x^1>\rho$, $t_k<1/2$, and
\begin{align*}
x^{k+1}-\rho=x^k-\rho-2t_k(x^k-\rho)=(1-2t_k)(x^k-\rho).
\end{align*} 
In addition, we readily get
\begin{align}\label{sq the}
0\le x^{k+1}-\rho=(1-2t_k)(x^k-\rho)=\ldots=(x^1-\rho) \prod_{i=1}^k(1-2t_i) .
\end{align}
Furthermore, deduce from $\sum_{k=1}^\infty 2t_k=\infty$ that $\prod_{k=1}^\infty(1-2t_k)=0.$ Indeed, we have
\begin{align*}
0\le \prod_{k=1}^\infty(1-2t_k)\le \frac{1}{\prod_{k=1}^\infty(1+2t_k)}\le \frac{1}{1+\sum_{k=1}^\infty 2t_k}=0.
\end{align*}
This tells us by \eqref{sq the} and the classical squeeze theorem that $x^k\rightarrow \rho$ as $k\to\infty$.
\end{proof}

\section{Auxiliary Results for Convergence Analysis}

We first establish the new three sequences lemma, which is crucial in the analysis of both SAM, USAM, and their variants.
\begin{lemma}[three sequences lemma]\rm\label{three sequences lemma}
   Let $\{\alpha_k\},\{\beta_k\}, \{\gamma_k\}$ be sequences of nonnegative numbers satisfying the conditions
   \begin{align}
   &\alpha_{k+1}-\alpha_k\le  \beta_k\alpha_k+ \gamma_k\tag{a}\;\text{ for sufficient large }\;k\in\N,\label{a}\\
   &\set{\beta_k}\text{ is bounded},\sum_{k=1}^\infty \beta_k=\infty,\; \sum_{k=1}^\infty \gamma_k<\infty,\;\mbox{ and }\;\sum_{k=1}^\infty \beta_k\alpha_k^2<\infty.\tag{b}   \label{b}
   \end{align}
Then we have that $\alpha_k\rightarrow0$ as $k\to\infty$.
\end{lemma}
\begin{proof}
First we show that $\liminf_{k\rightarrow\infty}\alpha_k=0.$ Supposing the contrary gives us some $\delta>0$ and $N\in \N$ such that $\alpha_k\ge \delta$ for all $k\ge N$. Combining this with the second and the third condition in \eqref{b} yields
\begin{align*}
\infty>\sum_{k=N}^\infty \beta_k\alpha_k^2\ge \delta^2 \sum_{k=N}^\infty \beta_k=\infty,
\end{align*}
which is a contradiction verifying the claim.
Let us now show that in fact $\lim_{k\rightarrow\infty} \alpha_k=0.$ Indeed, by the boundedness of $\set{\beta_k}$ define $\bar\beta:=\sup_{k\in\N}\beta_k$ and deduce from \eqref{a} that there exists $K\in\N$ such that 
\begin{align}\label{alpha con beta}
\alpha_{k+1}-\alpha_k\le \beta_k\alpha_k+ \gamma_k\text{ for all }k\ge K.
\end{align}
Pick $\varepsilon>0$ and find by $\liminf_{k\rightarrow\infty}\alpha_k=0$ and the two last conditions in \eqref{b} some $K_\varepsilon\in\N$ with $K_\varepsilon\ge K$, $\alpha_{K_\varepsilon}\le \varepsilon$,
\begin{align}\label{1 by 18}
\sum_{k=K_\varepsilon}^\infty \gamma_k<\frac{\varepsilon}{3}, \sum_{k=K_\varepsilon}^\infty \beta_k\alpha_k^2<\frac{\varepsilon^2 }{3},\;\text{ and }\; \bar\beta\beta_k\alpha_k^2\le \frac{\varepsilon^2}{9}\;\text{ as }\;k\ge K_\varepsilon.
\end{align}
It suffices to show that $\alpha_k\le 2\varepsilon$ for all $k\ge K_\varepsilon.$ Fix $k\ge K_\varepsilon$ and observe that for $\alpha_k\le \varepsilon$ the desired inequality is obviously satisfied. If $\alpha_k>\varepsilon,$ we use $\alpha_{K_\varepsilon}\le \varepsilon$ and find some $k'<k$ such that  $k'\ge K_\varepsilon$ and 
\begin{align*}
  \alpha_{k'}\le \varepsilon\;\text{ and }\;\alpha_i >\varepsilon\;\text{ for }\; i=k,k-1,\ldots,k'+1.
\end{align*}
Then we deduce from \eqref{alpha con beta} and \eqref{1 by 18} that
\begin{align*}
   \alpha_k- \alpha_{k'}&= \sum_{i=k'}^{k-1}\brac{\alpha_{i+1}-\alpha_i}\le \sum_{i=k'}^{k-1}\brac{ \beta_i\alpha_i+ \gamma_i}\\
   &=\sum_{i=k'+1}^{k} \beta_i\alpha_i+\sum_{i=k'}^{k-1} \gamma_i+\beta_{k'}\alpha_{k'} \\
&\le \frac{1}{\varepsilon}\sum_{i=k'+1}^{k}\beta_i \alpha_i^2+ \sum_{i=k'}^{k-1}\gamma_i+\sqrt{\beta_{k'}}\sqrt{\beta_{k'}}\alpha_{k'}\\
&\le \frac{1}{\varepsilon}\sum_{i=K_\varepsilon}^{\infty}\beta_i \alpha_i^2+ \sum_{i=K_\varepsilon}^{\infty}\gamma_i+\sqrt{\bar \beta \beta_{k'}\alpha_{k'}^2}\\
&\le \frac{1}{\varepsilon}\frac{\varepsilon^2}{3}+\frac{\varepsilon}{3}+\frac{\varepsilon}{3}=\varepsilon.
\end{align*}
As a consequence, we arrive at the estimate
\begin{align*}
\alpha_k= \alpha_{k'}+\alpha_k -\alpha_{k'}\le \varepsilon+\varepsilon=2\varepsilon\;\mbox{ for all }\;k\ge K_\ve, 
\end{align*}
which verifies that $\alpha_k\rightarrow 0$ as $k\to\infty$ sand thus completes the proof of the lemma.
\end{proof}

Next we recall some auxiliary results from \citet{kmt23.1}.

\begin{lemma}\label{stationary point lemma}
Let $\set{x^k}$ and $\set{d^k}$ be sequences in $\R^n$ satisfying the condition
\begin{align*}
\sum_{k=1}^\infty \norm{x^{k+1}-x^k}\cdot\norm{d^k}<\infty.
\end{align*}
If $\bar x$ is an accumulation point of the sequence $\set{x^k}$ and $0$ is an accumulation points of the sequence $\set{d^k}$, then there exists an infinite set $J\subset\N$ such that we have
\begin{align}\label{relation dk xk 3}
x^k\overset{J}{\rightarrow}\bar x\;\mbox{ and }\;d^k\overset{J}{\rightarrow}0.
\end{align}
\end{lemma}

\begin{proposition}\label{general convergence under KL}
Let $f:\R^n\rightarrow\R$ be a $\mathcal{C}^1$-smooth function, and let the sequence $\set{x^k}\subset\R^n$ satisfy the conditions:
\begin{itemize}
\item[\bf(H1)] {\rm(primary descent condition)}. There exists $\sigma>0$ such that for sufficiently large $k\in\N$ we have 
\begin{align*}
 f(x^k)-f(x^{k+1})\ge\sigma\norm{\nabla f(x^k)}\cdot\norm{x^{k+1}-x^k}.
\end{align*}
\item[\bf(H2)] {\rm(complementary descent condition)}. For sufficiently large $k\in\N$, we have
\begin{align*}
 \big[f(x^{k+1})=f(x^k)\big]\Longrightarrow [x^{k+1}=x^k].
\end{align*}
\end{itemize}
If $\bar x$ is an accumulation point of $\set{x^k}$  and $f$ satisfies the KL property at $\bar x$, then $x^k\rightarrow\bar x$ as $k\to\infty$.
\end{proposition}
When the sequence under consideration is generated by a linesearch method and satisfies some conditions stronger than (H1) and (H2) in 
Proposition~\ref{general convergence under KL}, its convergence rates  are established in \citet[Proposition~2.4]{kmt23.1} under the KL property with $\psi(t)=Mt^{1-q}$ as given below.
\begin{proposition}\label{general rate}
Let $f:\R^n\rightarrow\R$ be a $\mathcal{C}^1$-smooth function, and let the sequences $\set{x^k}\subset\R^n,\set{\tau_k}\subset[0,\infty),\set{d^k}\subset\R^n$ satisfy the iterative condition $x^{k+1}=x^k+\tau_kd^k$ for all $k\in\N$. Assume that for all sufficiently large $k\in\N$ we have $x^{k+1}\ne x^k$ and the estimates
\begin{align}\label{two conditions}
f(x^k)-f(x^{k+1})\ge \beta \tau_k\norm{d^k}^2\;\text{ and }\;\norm{\nabla f(x^k)}\le \alpha\norm{d^k},
\end{align}
where $\alpha,\beta>0$. Suppose in addition that the sequence $\set{\tau_k}$ is bounded away from $0$ $($i.e., there is some $\bar \tau>0$ such that $\tau_k\ge \bar \tau$ for large $k\in\N)$, that $\bar x$ is an accumulation point of $\set{x^k}$, and that $f$ satisfies the KL property at $\bar x$ with $\psi(t)=Mt^{1-q}$ for some $M>0$ and $q\in(0,1)$. Then the following convergence rates are guaranteed:
\begin{itemize}\vspace*{-0.05in}
\item[\bf(i)] If $q\in(0,1/2]$, then the sequence $\set{x^k}$ converges linearly to $\bar x$.\vspace*{-0.05in}
\item[\bf(ii)]If $q\in(1/2,1)$, then we have the estimate
\begin{align*}
\norm{x^k-\bar x}=\mathcal{O}\brac{ k^{-\frac{1-q}{2q-1}}}.
\end{align*}
\end{itemize}
\end{proposition}

Yet another auxiliary result needed below is as follows.

\begin{proposition} \label{convergence rate deduce}
Let $f:\R^n\rightarrow\R$ be a ${\cal C}^1$-smooth function satisfying the descent condition \eqref{descent condition} with some constant $L>0$. Let $\set{x^k}$ be a sequence in $\R^n$ that converges to $\bar x$, and let $\alpha>0$ be such that 
\begin{align}\label{main condition rate}
\alpha\norm{\nabla f(x^k)}^2\le f(x^k)-f(x^{k+1})\;\text{ for sufficiently large }\;\in\N.
\end{align}
Consider the following convergence rates of $\set{x^k}:$
\begin{itemize}
\item[\bf(i)] $x^k\rightarrow\bar x$ linearly.
\item[\bf(ii)] $\norm{x^k-\bar x}=\mathcal{O}(m(k))$, where $m(k)\downarrow0$ as $k\rightarrow\infty.$
\end{itemize}
Then {\rm(i)} ensures the linear convergences of $f(x^k)$ to $f(\bar x)$, and $\nabla f(x^k)$ to $0$, while {\rm(ii)} yields $\abs{f(x^k)-f(\bar x)}=\mathcal{O}(m^2(k))$ and $\norm{\nabla f(x^k)}=\mathcal{O}(m(k))$ as $k\to\infty$. 
\end{proposition}
\begin{proof}
Condition \eqref{main condition rate} tells us that there exists some $N\in\N$ such that $f(x^{k+1})\le f(x^k)$ for all $k\ge \N.$ As $x^k\rightarrow\bar x$, we deduce that $f(x^k)\rightarrow f(\bar x)$ with $f(x^k)\ge f(\bar x)$ for $k\ge N.$ Letting $k\rightarrow\infty$ in \eqref{main condition rate} and using the squeeze theorem together with the convergence of $\set{x^k}$ to $\bar x$ and the continuity of $\nabla f$ lead us to $\nabla f(\bar x)=0.$ It follows from the descent condition of $f$ with constant $L>0$ and from \eqref{descent condition} that 
\begin{align*}
0\le f(x^k)-f(\bar x)\le \dotproduct{\nabla f(\bar x),x^k-\bar x}+\dfrac{{L}}{2}\norm{x^k-\bar x}^2=\dfrac{{L}}{2}\norm{x^k-\bar x}^2.
\end{align*}
This verifies the desired convergence rates of $\set{f(x^k)}.$
Employing finally \eqref{main condition rate} and $f(x^{k+1})\ge f(\bar x)$, we also get that 
\begin{align*}
   \alpha \norm{\nabla f(x^k)}^2\le f(x^k)-f(\bar x)\text{ for all }k\ge N.
\end{align*}
This immediately gives us the desired convergence rates for $\set{\nabla f(x^k)}$ and completes the proof.
\end{proof}

\section{Proof of Convergence Results}
\subsection{Proof of Theorem~\ref{theo SAM constant}}\label{appendix theo SAM constant}
\begin{proof}
To verify (i) first, for any $k\in\N$ define $g^k:=\nabla f\brac{x^k+\rho_k\frac{\nabla f(x^k)}{\norm{\nabla f(x^k)}}}$ and get
\begin{align}\label{Lrho}
\norm{g^k-\nabla f(x^k)}&=\norm{\nabla f\brac{x^k+\rho_k\frac{\nabla f(x^k)}{\norm{\nabla f(x^k)}}}-\nabla f(x^k)}\nonumber\\
&\le L \norm{x^k+\rho_k\frac{\nabla f(x^k)}{\norm{\nabla f(x^k)}}-x^k}=L\rho_k\le L\rho,
\end{align}
where $\rho:=\sup_{k\in\N}\rho_k.$ Using the monotonicity of $\nabla f$ due to the convexity of $f$ ensures that
\begin{align}\label{dot esti}
\dotproduct{g^k,\nabla f(x^k)}&=\dotproduct{\nabla f\brac{x^k+\rho_k\frac{\nabla f(x^k)}{\norm{\nabla f(x^k)}}}-\nabla f(x^k),\nabla f(x^k)}+\norm{\nabla f(x^k)}^2\nonumber\\
&=\frac{\norm{\nabla f(x^k)}}{\rho_k}\dotproduct{\nabla f\brac{x^k+\rho_k\frac{\nabla f(x^k)}{\norm{\nabla f(x^k)}}}-\nabla f(x^k),\rho_k\frac{\nabla f(x^k)}{\norm{\nabla f(x^k)}}}+\norm{\nabla f(x^k)}^2\ge \norm{\nabla f(x^k)}^2.
\end{align}
With the definition of $g^k,$ the iterative procedure \eqref{iterative SAM} can also be rewritten as $x^{k+1}=x^k-t_kg^k$ for all $k\in\N$. The first condition in \eqref{parameter constant rho} yields $t_k\downarrow 0,$ which gives us some $K\in\N$ such that $Lt_k<1$ for all $k\ge K$. Take some such $k$. Since $\nabla f$ is Lipschitz continuous with constant $L>0$, it follows from the descent condition in  \eqref{descent condition} and the estimates in \eqref{Lrho}, \eqref{dot esti} that
\begin{align}\label{first esti}
f(x^{k+1})&\le f(x^k)+\dotproduct{\nabla f(x^k),x^{k+1}-x^k}+\frac{L}{2}\norm{x^{k+1}-x^k}^2\nonumber\\
&=f(x^k)-t_k\dotproduct{\nabla f(x^k),g^k}+\frac{Lt_k^2}{2}\norm{g^k}^2\nonumber\\
&= f(x^k)-t_k(1-Lt_k) \dotproduct{\nabla f(x^k),g^k}+\frac{Lt_k^2}{2}\brac{\norm{g^k-\nabla f(x^k)}^2-\norm{\nabla f(x^k)}^2}\nonumber\\
&\le f(x^k) -t_k(1-Lt_k)\norm{\nabla f(x^k)}^2-\frac{Lt_k^2}{2}\norm{\nabla f(x^k)}^2+\frac{L^3t_k^2\rho^2}{2}\nonumber\\
&=f(x^k)-\frac{t_k}{2}\brac{2-Lt_k}\norm{\nabla f(x^k)}^2+\frac{L^3t_k^2\rho^2}{2}\nonumber\\
&\le f(x^k)-\frac{t_k}{2}\norm{\nabla f(x^k)}^2+\frac{L^3t_k^2\rho^2}{2}.
\end{align} 
Rearranging the terms above gives us the estimate
\begin{align}\label{estimate convex}
\frac{t_k}{2}\norm{\nabla f(x^k)}^2\le f(x^{k})-f(x^{k+1})+\frac{L^3t_k^2\rho^2}{2}.
\end{align}
Select any $M>K,$ define $S:=\frac{L^3\rho^2}{2}\sum_{k=1}^\infty t_k^2<\infty$, and get by taking into account $\inf_{k\in\N} f(x^k)>-\infty$ that
\begin{align*}
\frac{1}{2}\sum_{k=K}^M t_k\norm{\nabla f(x^k)}^2&\le \sum_{k=K}^M\brac{f(x^{k})-f(x^{k+1})}+\sum_{k=K}^M\frac{L^3t_k^2\rho^2}{2}\\
&\le f(x^K)-f(x^{M+1})+S\\
&\le f(x^K)-\inf_{k\in\N} f(x^k)+S.
\end{align*}
Letting $M\rightarrow\infty$ yields $\sum_{k=K}^\infty t_k\norm{\nabla f(x^k)}^2<\infty.$ Let  us now show that $\liminf\norm{\nabla f(x^k)}=0.$ Supposing the contrary gives us $\varepsilon>0$ and $N\ge K$ such that $\norm{\nabla f(x^k)}\ge \varepsilon$ for all $k\ge N$, which tells us that
\begin{align*}
\infty>\sum_{k=N}^\infty t_k\norm{\nabla f(x^k)}^2\ge \varepsilon^2 \sum_{k=N}^\infty t_k=\infty.
\end{align*}
This clearly contradicts the second condition in \eqref{parameter constant rho} and this justifies (i).
 
To verify (ii), define $ u_k:=\frac{L^3\rho^2}{2}\sum_{i=k}^\infty t_i^2\;\text{ for all }\;k\in\N$ and deduce from the first condition in \eqref{parameter constant rho} that $u_k\downarrow 0$ as $k\rightarrow\infty.$ With the usage of $\set{u_k},$ estimate \eqref{estimate convex} is written as
\begin{align*}
f(x^{k+1})+u_{k+1}\le f(x^k)+u_k-\frac{t_k}{2}\norm{\nabla f(x^k)}^2\; \text{ for all }\;k\ge K,
\end{align*}
which means that $\set{f(x^{k})+u_k}_{k\ge K}$ is nonincreasing. It follows from $\inf_{k\in\N} f(x^k)>-\infty$ and $u_k\downarrow0$ that $\set{f(x^k)+u_k}$ is convergent, which means that the sequence $\set{f(x^k)}$ is convergent as well. Assume now $f$ has some nonempty and bounded level set. Then every level set of $f$ is bounded by \citet[Exercise~2.12]{rusbook}. By \eqref{estimate convex}, we get that
\begin{align*}
f(x^{k+1})\le f(x^k)+\frac{L^3\rho^2}{2}t_k^2-\frac{t_k}{2}\norm{\nabla f(x^k)}^2\le f(x^k)+\frac{L^3\rho^2}{2}t_k^2\;\text{ for all }\;k\ge K.
\end{align*}
Proceeding by induction leads us to
\begin{align*}
f(x^{k+1})\le f(x^1)+\frac{L^3\rho^2}{2}\sum_{i=1}^k t_i^2\le f(x^1)+S\;\text{ for all }\;k\ge K,
\end{align*}
which means that $x^{k+1}\in \set{x\in \R^n\;|\;f(x)\le f(x^1)+S}$ for all $k\ge K$ and thus justifies the boundedness of $\set{x^k}$. 

Taking $\liminf \norm{\nabla f(x^k)}=0$ into account gives us an infinite set $J\subset \N$ such that $\norm{\nabla f(x^k)}\overset{J}{\rightarrow}0.$ As $\set{x^k}$ is bounded, the sequence $\set{x^{k}}_{k\in J}$ is also bounded, which gives us another infinite set $I\subset J$ and $\bar x\in\R^n$ such that $x^k\overset{I}{\rightarrow}\bar x.$ By 
\begin{align*}
\lim_{k\in I}\norm{\nabla f(x^k)}=\lim_{k\in J}\norm{\nabla f(x^k)}=0
\end{align*}
and the continuity of $\nabla f$, we get that $\nabla f(\bar x)=0$ ensuring that $\bar x$ is a global minimizer of $f$ with the optimal value $f^*:=f(\bar x).$ Since the sequence $\set{f(x^k)}$ is convergent and since $\bar x$ is an accumulation point of $\set{x^k}$, we conclude that $f^*=f(\bar x)$ is the limit of $\set{f(x^k)}$. Now take any accumulation point $\tilde x$ of $\set{x^k}$ and find an infinite set $J'\subset \N$ with $x^k\overset{J'}{\rightarrow}\tilde x.$ As $\set{f(x^k)}$ converges to $f^*$, we deduce that 
\begin{align*}
f(\tilde x)=\lim_{k\in J'}f(x^k)=\lim_{k\in \N}f(x^k)=f^*,
\end{align*}
which implies that $\tilde x$ is also a global minimizer of $f.$ Assuming in addition that $f$ has a unique minimizer $\bar{x}$ and taking any accumulation point $\tilde x$ of $\set{x^k},$ we get that $\tilde x$ is a minimizer of $f,$ i.e., $\tilde x=\bar x.$ This means that $\bar x$ is the unique accumulation point of $\set{x^k}$, and therefore $x^k\rightarrow\bar x$ as $k\to\infty$.
\end{proof}

\subsection{Proof of Theorem~\ref{convergence IGD}}\label{appendix convergence IGD}

\begin{proof} By \eqref{parameter general}, we find some $c_1>0,c_2\in(0,1)$, and $K\in\N$ such that 
\begin{align}\label{defi c1c2}
\frac{1}{2}( 2-Lt_k- \varepsilon_k+L t_k\varepsilon_k)\ge c_1,\quad \frac{1}{2}(1-Lt_k)+\frac{Lt_k\varepsilon_k}{2}\le  c_2,\;\mbox{ and}\quad Lt_k<1\;\text{ for all }\;k\ge K.
\end{align}
Let us first verify the estimate
\begin{align}\label{non descent}
f(x^{k+1})\le f(x^k) -c_1 t_k\norm{\nabla f(x^k)}^2+c_2t_k\varepsilon_k\;\text{ whenever }\;k\ge K.
\end{align}
To proceed, fix $k\in\N$ and deduce from the Cauchy-Schwarz inequality that
\begin{align}\label{esti product}
\dotproduct{g^k,\nabla f(x^k)}&=\dotproduct{g^k-\nabla f(x^k),\nabla f(x^k)}+\norm{\nabla f(x^k)}^2\nonumber\\
&\ge -\norm{g^k-\nabla f(x^k)}\cdot\norm{\nabla f(x^k)}+\norm{\nabla f(x^k)}^2\nonumber\\
&\ge -\varepsilon_k\norm{\nabla f(x^k)}+\norm{\nabla f(x^k)}^2.
\end{align}
Since $\nabla f$ is Lipschitz continuous with constant $L$, it follows from the descent condition in \eqref{descent condition} and the estimate \eqref{esti product} that
\begin{align*}
f(x^{k+1})&\le f(x^k)+\dotproduct{\nabla f(x^k),x^{k+1}-x^k}+\frac{L}{2}\norm{x^{k+1}-x^k}^2\nonumber\\
&=f(x^k)-t_k\dotproduct{\nabla f(x^k),g^k}+\frac{Lt_k^2}{2}\norm{g^k}^2\nonumber\\
&=f(x^k)-t_k(1-Lt_k)\dotproduct{\nabla f(x^k),g^k}+\frac{Lt_k^2}{2}(\norm{g^k-\nabla f(x^k)}^2-\norm{\nabla f(x^k)}^2)\nonumber\\
&\le f(x^k)-t_k(1-Lt_k)\brac{-\varepsilon_k\norm{\nabla f(x^k)}+\norm{\nabla f(x^k)}^2}+\frac{Lt_k^2\varepsilon_k^2}{2}-\frac{Lt_k^2}{2}\norm{\nabla f(x^k)}^2\nonumber\\
&=f(x^k)-\frac{t_k}{2}(2-{Lt_k})\norm{\nabla f(x^k)}^2+t_k(1-Lt_k)\varepsilon_k\norm{\nabla f(x^k)}+\frac{Lt_k^2\varepsilon_k^2}{2}\nonumber\\
&\le f(x^k)-\frac{t_k}{2}(2-{Lt_k})\norm{\nabla f(x^k)}^2+\frac{1}{2} t_k(1-Lt_k)\varepsilon_k\brac{1+\norm{\nabla f(x^k)}^2}+\frac{Lt_k^2\varepsilon_k^2}{2}\nonumber\\
&=f(x^k)-\frac{t_k}{2}(2-Lt_k-\varepsilon_k+L t_k\varepsilon_k)\norm{\nabla f(x^k)}^2+\frac{1}{2} t_k\varepsilon_k(1-Lt_k)+\frac{Lt_k^2\varepsilon_k^2}{2}\nonumber\\
&=f(x^k)-\frac{t_k}{2}(2-Lt_k-\varepsilon_k+L t_k\varepsilon_k)\norm{\nabla f(x^k)}^2+t_k\varepsilon_k\brac{\frac{1}{2} (1-Lt_k)+\frac{Lt_k\varepsilon_k}{2}.}\nonumber
\end{align*} 
Combining this with \eqref{defi c1c2} gives us \eqref{non descent}. Defining $u_k:=c_2\sum_{i=k}^\infty t_i\varepsilon_i$ for $k\in\N$, we get that $u_k\rightarrow 0$ as $k\rightarrow\infty$ and $u_k-u_{k+1}=t_k\varepsilon_k$ for all $k\in\N.$ Then \eqref{non descent} can be rewritten as
\begin{align}\label{descent f xk uk}
f(x^{k+1})+u_{k+1}\le f(x^k)+u_k-c_1t_k\norm{\nabla f(x^k)}^2,\quad k\ge K.
\end{align}
To proceed now with the proof of (i), we deduce from \eqref{descent f xk uk} combined with $\inf f(x^k)>-\infty$ and $u_k\rightarrow0$ as $k\rightarrow\infty$ that
\begin{align*}
c_1\sum_{k=K}^\infty t_k\norm{\nabla f(x^k)}^2&\le \sum_{k=K}^\infty (f(x^k)-f(x^{k+1})+u_k-u_{k+1} )\\
&\le f(x^K)-\inf_{k\in\N} f(x^k)+u_K<\infty.
\end{align*}
Next we employ Lemma~\ref{three sequences lemma} with $\alpha_k:=\norm{\nabla f(x^k)}$, $\beta_k:=Lt_k$, and $\gamma_k:=Lt_k\varepsilon_k$ for all $k\in\N$ to derive $\nabla f(x^k)\rightarrow0.$ Observe first that condition \eqref{a} is satisfied due to the the estimates
\begin{align*}
{\alpha_{k+1}-\alpha_k}&={\norm{\nabla f(x^{k+1})}-\norm{\nabla f(x^{k})}}\le \norm{\nabla f(x^{k+1})-\nabla f(x^k)}\\
&\le L\norm{x^{k+1}-x^k}=Lt_k\norm{g^k}\\
&\le Lt_k (\norm{\nabla f(x^k)}+\norm{g^k-\nabla f(x^k)})\\
&\le Lt_k (\norm{\nabla f(x^k)}+\varepsilon_k)\\
&=\beta_k \alpha_k+\gamma_k\;\text{ for all }\;k\in\N.
\end{align*}
Further, the conditions in \eqref{b} hold by \eqref{parameter general} and $\sum_{k=1}^\infty t_k\norm{\nabla f(x^k)}^2<\infty.$ As all the assumptions \eqref{a}, \eqref{b} are satisfied, Lemma~\ref{three sequences lemma} tells us that $\norm{\nabla f(x^k)}=\alpha_k\rightarrow0$ as $k\rightarrow\infty.$

To verify (ii), deduce from \eqref{descent f xk uk} that $\set{f(x^k)+u_k}$ is nonincreasing. As $\inf_{k\in\N} f(x^k)>-\infty$ and $u_k\rightarrow0,$ we get that $\set{f(x^k)+u_k}$ is bounded from below, and thus is convergent. Taking into account that $u_k\rightarrow0$, it follows that $f(x^k)$ is convergent as well. Since $\bar x$ is an accumulation point of $\set{x^k},$ the continuity of $f$ tells us that $f(\bar x)$ is also an accumulation point of $\set{f(x^k)},$ which immediately yields $f(x^k)\rightarrow f(\bar x)$ due to the convergence of $\set{f(x^k)}.$ 

It remains to verify (iii). By the KL property of $f$ at $\bar x,$ we find some $\eta>0,$ a neighborhood $U$ of $\bar x$, and a desingularizing concave continuous function $\varphi:[0,\eta)\rightarrow[0,\infty)$ such that $\varphi(0)=0$, $\varphi$ is $\mathcal{C}^1$-smooth on $(0,\eta)$, $\varphi'>0$ on $(0,\eta)$, and we have for all $x\in U$ with $0<f(x)-f(\bar x)<\eta$ that
\begin{align}
\varphi'(f(x)-f(\bar x))\norm{\nabla f(x)}\ge 1.
\end{align}
Let $\bar K>K$ be natural number such that $f(x^k)> f(\bar x)$ for all $k\ge \bar K$. Define $\Delta_k:=\varphi(f(x^k)-f(\bar x)+u_k)$ for all $k\ge \bar K$, and let $R>0$ be such that $\B(\bar x, R)\subset U$. Taking the number $C$ from  Assumption~\ref{assu desi}, remembering that $\bar x$ is an accumulation point of $\set{x^k},$ and using $f(x^k)+u_k\downarrow f(\bar x)$, $\Delta_k\downarrow0$  as $k\rightarrow\infty$ together with  condition \eqref{desing condi}, we get  
by choosing a larger $\bar K$ that $f(x^{\bar K})+u_{\bar K}<f(\bar x)+\eta$ and
\begin{align}\label{finding k1}\norm{x^{\bar K}-\bar x}+\frac{1}{{Cc_1}}\Delta_{\bar K}+\sum_{k={\bar K}}^\infty t_k\varphi'\brac{\sum_{i=k}^\infty t_i\varepsilon_i}^{-1}+\sum_{k=\bar K}^\infty t_k\varepsilon_k<R.
\end{align}
Let us now show by induction that $x^k\in \B(\bar x,R)$ for all $k\ge \bar K.$ The assertion obviously holds for $k=\bar K$ due to \eqref{finding k1}. Take some $\widehat{K}\ge\bar K$ and suppose that $x^k\in U$ for all $k=\bar K,\ldots,\widehat{K}.$ We intend to show that $x^{\widehat{K}+1}\in \B(\bar x,R)$ as well. To proceed, fix some $k\in \set{\bar K,\ldots,\widehat{K}}$ and get by $f(\bar x)<f(x^k)< f(x^k)+u_k< f(\bar x)+\eta$ that
\begin{align}\label{sup for eq4}
\varphi'(f(x^k)-f(\bar x))\norm{\nabla f(x^k)}\ge 1.
\end{align}
Combining this with $u_k>0$ and $f(x^k)-f(\bar x)>0$ gives us
\begin{subequations}
\begin{align}
\Delta_k-\Delta_{k+1}&\ge \varphi'(f(x^k)-f(\bar x)+u_k)(f(x^k)+u_k-f(x^{k+1})-u_{k+1})\label{eq1}\\
&\ge \varphi'(f(x^k)-f(\bar x)+u_k)c_1t_k\norm{\nabla f(x^k)}^2 \label{eq2}\\
&\ge \frac{C}{\brac{\varphi'(f(x^k)-f(\bar x))}^{-1}+\brac{\varphi'(u_k)}^{-1}} c_1t_k\norm{\nabla f(x^k)}^2\label{eq3}\\
&\ge \frac{C}{\norm{\nabla f(x^k)}+\brac{\varphi'(u_k)}^{-1}} c_1t_k\norm{\nabla f(x^k)}^2\label{eq4},
\end{align}
\end{subequations}
where \eqref{eq1} follows from the concavity of $\varphi$, \eqref{eq2} follows from \eqref{descent f xk uk}, \eqref{eq3} follows from Assumption~\ref{assu desi}, and \eqref{eq4} follows from \eqref{sup for eq4}. Taking the square root of both sides in \eqref{eq4} and employing the AM-GM inequality yield
\begin{align}\label{eq5}
  t_k\norm{\nabla f(x^k)}= \sqrt{t_k}\cdot \sqrt{t_k\norm{\nabla f(x^k)}^2}&\le \sqrt{\frac{1}{Cc_1}(\Delta_{k}-\Delta_{k+1})t_k(\norm{\nabla f(x^k)}+\brac{\varphi'(u_k)}^{-1}})\nonumber\\
&\le \frac{1}{2Cc_1}(\Delta_k-\Delta_{k+1})+\frac{1}{2}t_k\brac{\brac{\varphi'(u_k)}^{-1}+\norm{\nabla
 f(x^k)}}.
\end{align}
Using the nonincreasing property of $\varphi'$ due to the concavity of $\varphi$ and the choice of $c_2\in (0,1)$ ensures that 
\begin{align*}
\brac{\varphi'(u_k)}^{-1}=\brac{\varphi'(c_2\sum_{i=k}^\infty t_i\varepsilon_i)}^{-1}\le \brac{\varphi'(\sum_{i=k}^\infty t_i\varepsilon_i)}^{-1}.
\end{align*}
Rearranging terms and taking the sum over $k=\bar K, \ldots, \widehat K$ of inequality \eqref{eq5} gives us
\begin{align*}
\sum_{k={\bar K}}^{\widehat{K}} t_k\norm{\nabla f(x^k)}&\le \frac{1}{{Cc_1}}\sum_{k={\bar K}}^{\widehat K} (\Delta_{k}-\Delta_{k+1})+ \sum_{k={\bar K}}^{\widehat K} t_k\varphi'(u_{k})^{-1}\\
&= \frac{1}{{Cc_1}}(\Delta_{\bar K}-\Delta_{\widehat{K}})+\sum_{k={\bar K}}^{\widehat{K}} t_k\varphi'\brac{c_2\sum_{i=k}^\infty t_i\varepsilon_{i}}^{-1}\\
&\le \frac{1}{{Cc_1}} \Delta_{\bar K}+\sum_{k=\bar K}^{\widehat K} t_k\varphi'\brac{\sum_{i=k}^\infty t_i\varepsilon_{i}}^{-1}.
\end{align*}
The latter estimate together with the triangle inequality and \eqref{finding k1} tells us that
\begin{align*}
\norm{x^{\widehat K+1}-\bar x}&=\norm{x^{\bar K}-\bar x}+\sum_{k={\bar K}}^{\widehat K}\norm{x^{k+1}-x^k}\\
&=\norm{x^{\bar K}-\bar x}+\sum_{k={\bar K}}^{\widehat K}t_k\norm{g^k}\\
&\le \norm{x^{\bar K}-\bar x}+\sum_{k={\bar K}}^{\widehat K}t_k\norm{\nabla f(x^k)}+\sum_{k={\bar K}}^{\widehat K}t_k\norm{g^k-\nabla f(x^k)}\\
&\le \norm{x^{\bar K}-\bar x}+\sum_{k={\bar K}}^{\widehat K}t_k\norm{\nabla f(x^k)}+\sum_{k={\bar K}}^{\widehat K}t_k\varepsilon_{k}\\
&\le \norm{x^{\bar K}-\bar x}+\frac{1}{Cc_1} \Delta_{\bar K}+\sum_{k={\bar K}}^{\infty} t_k\varphi'\brac{\sum_{i=k}^\infty t_i\varepsilon_{i}}^{-1}+\sum_{k={\bar K}}^\infty t_k\varepsilon_{k}<R.
\end{align*}
By induction, this means that $x^k\in\B(\bar x,R)$ for all $k\ge {\bar K}.$ Then a similar device brings us to
\begin{align*}
\sum_{k=\bar K}^{\widehat K}t_k\norm{\nabla f(x^k)}\le \frac{1}{{Cc_1}} \Delta_{\bar K}+ \sum_{k={\bar K}}^{\infty} t_k\varphi'\brac{\sum_{i=k}^\infty t_i\varepsilon_{i}}^{-1}\text{ for all }\widehat K\ge \bar K,
\end{align*}
which yields $ \sum_{k=1}^\infty t_k\norm{\nabla f(x^k)}<\infty$. Therefore,
\begin{align*}
  \sum_{k=1}^\infty \norm{x^{k+1}-x^k}&= \sum_{k=1}^\infty t_k\norm{g^k}\le  \sum_{k=1}^\infty t_k\norm{\nabla f(x^k)}+ \sum_{k=1}^\infty t_k\norm{g^k-\nabla f(x^k)}\\
  &\le\sum_{k=1}^\infty t_k\norm{\nabla f(x^k)}+ \sum_{k=1}^\infty t_k\varepsilon_{k}<\infty
\end{align*}
which justifies the convergence of $\set{x^k}$ and thus completes the proof of the theorem.
\end{proof}

\subsection{Proof of Theorem~\ref{theo IGDr}}\label{appendix theo IGDr}

\begin{proof} Using $\norm{\nabla f(x^k)-g^k}\le \nu\norm{\nabla f(x^k)}$ gives us the estimates
\begin{align}
\norm{g^k}^2&=\norm{\nabla f(x^k)-g^k}^2 -\norm{\nabla f(x^k)}^2+2\dotproduct{\nabla f(x^k),g^k}\nonumber\\
&\le \nu^2 \norm{\nabla f(x^k)}^2-\norm{\nabla f(x^k)}^2+2\dotproduct{\nabla f(x^k),g^k}\nonumber\\
&= -(1-\nu^2)\norm{\nabla f(x^k)}^2+2\dotproduct{\nabla f(x^k),g^k},\label{es 1}\\
\dotproduct{\nabla f(x^k),g^k} &=\dotproduct{\nabla f(x^k),g^k-\nabla f(x^k)}+\norm{\nabla f(x^k)}^2\nonumber\\
&\le \norm{\nabla f(x^k)}\cdot\norm{g^k-\nabla f(x^k)}+\norm{\nabla f(x^k)}^2\nonumber\\
&\le(1+\nu)\norm{\nabla f(x^k)}^2,\label{es 2}\\
-\dotproduct{\nabla f(x^k),g^k} &=-\dotproduct{\nabla f(x^k),g^k-\nabla f(x^k)}-\norm{\nabla f(x^k)}^2\nonumber\\
&\le \norm{\nabla f(x^k)}\cdot\norm{g^k-\nabla f(x^k)}-\norm{\nabla f(x^k)}^2\nonumber\\
&\le -(1-\nu)\norm{\nabla f(x^k)}^2,\label{es 3}\\
 \norm{\nabla f(x^k)}-\norm{g^k-\nabla f(x^k)}   & \le \norm{g^k}\le \norm{\nabla f(x^k)}+\norm{g^k-\nabla f(x^k)},\nonumber
\end{align}
which in turn imply that
\begin{align}\label{two inequality}
(1-\nu )\norm{\nabla f(x^k)}\le \norm{g^k}\le(1+\nu) \norm{\nabla f(x^k)}\text{ for all }k\in\N.
\end{align}
Using condition \eqref{stepsize IGDr}, we find $N\in\N$ so that $2-2\nu -Lt_k(1+\nu)^2 \ge \delta$ for all $k\ge N.$ Select such a natural number $k$ and use the Lipschitz continuity of $\nabla f$ with constant $L$ to deduce from the descent condition (\ref{descent condition}), the relationship $x^{k+1}=x^k-t_kg^k$, and the estimates \eqref{es 1}--\eqref{es 3} that
\begin{align}
f(x^{k+1})&\le f(x^k)+\dotproduct{\nabla f(x^k),x^{k+1}-x^k}+\dfrac{L}{2} \norm{x^{k+1}-x^k}^2\nonumber\\
&= f(x^k)-t_k \dotproduct{\nabla f(x^k),g^k}+\dfrac{Lt_k^2}{2} \norm{g^k}^2\nonumber\\
&\le f(x^k)-t_k\dotproduct{\nabla f(x^k), g^k}+Lt_k^2\dotproduct{\nabla f(x^k), g^k}-\frac{Lt_k^2(1-\nu^2)}{2} \norm{\nabla f(x^k)}^2\nonumber\\
&\le f(x^k)-t_k(1-\nu)\norm{\nabla f(x^k)}^2+ Lt_k^2(1+\nu)\norm{\nabla f(x^k)}^2-\frac{Lt_k^2(1-\nu^2)}{2} \norm{\nabla f(x^k)}^2\nonumber\\
&=f(x^k)-\frac{t_k}{2}\brac{2-2\nu-Lt_k (1+\nu)^2} \norm{\nabla f(x^k)}^2\nonumber\\
&\le f(x^k)-\frac{\delta t_k}{2}\norm{\nabla f(x^k)}^2\label{descent esti 1}\text{ for all }k\ge N.
\end{align}
It follows from the above that the sequence $\set{f(x^k)}_{k\ge N}$ is nonincreasing, and hence the condition $\inf_{k\in\N} f(x^k)>-\infty$ ensures the convergence of $\set{f(x^k)}$. This allows us to deduce from \eqref{descent esti 1} that
\begin{align}\label{seri tk fxk con}
\frac{\delta}{2} \sum_{k=N}^\infty t_k\norm{\nabla f(x^k)}^2\le \sum_{K=N}^\infty (f(x^k)-f(x^{k+1}))\le f(x^K)-\inf_{k\in\N} f(x^k)<\infty.
\end{align}
Combining the latter with \eqref{two inequality} and $x^{k+1}=x^k-t_kg^k$ gives us 
\begin{align}\label{common lemma}
 \sum_{k=1}^\infty \norm{x^{k+1}-x^k}\cdot\norm{g^k}  =\sum_{k=1}^\infty t_k\norm{g^k}^2\le (1+\nu)^2\sum_{k=1}^\infty t_k\norm{\nabla f(x^k)}^2<\infty.
\end{align}
Now we are ready to verify all the assertions of the theorem. Let us start with (i) and show that $0$ in an accumulation point of $\set{g^k}$. Indeed, supposing the contrary gives us $\varepsilon>0$ and $K\in\N$ such that $\norm{g^k}\ge \varepsilon$ for all $k\ge K$, and therefore 
\begin{align*}
\infty>\sum_{k=K}^\infty t_k\norm{g^k}^2\ge \sum_{k=K}^\infty t_k=\infty,
\end{align*}
which is a contradiction justifying that $0$ is an accumulation point of $\set{g^k}$. If
$\bar x$ is an accumulation point of $\set{x^k},$ then by Lemma~\ref{stationary point lemma} and \eqref{common lemma}, we find an infinite set $J\subset N$ such that $x^k\overset{J}{\rightarrow}\bar x$ and $g^k\overset{J}{\rightarrow}0.$ The latter being combined with \eqref{two inequality} gives us $\nabla f(x^k)\overset{J}{\rightarrow}0$, which yields the stationary condition $\nabla f(\bar x)=0.$

To verity (ii), let $\bar x$ be an accumulation point of $\set{x^k}$ and find an infinite set $J\subset \N$ such that $x^k\overset{J}{\rightarrow}\bar x.$ Combining this with the continuity of $f$ and the fact that $\set{f(x^k)}$ is convergent, we arrive at the equalities
\begin{align*}
f(\bar x)= \lim_{k\in J}f(x^k)= \lim_{k\in \N}f(x^k),
\end{align*}
which therefore justify assertion (ii).

To proceed with the proof of the next assertion (iii), assume that $\nabla f$ is Lipschitz continuous with constant $L>0$ and employ Lemma~\ref{three sequences lemma} with $\alpha_k:=\norm{\nabla f(x^k)}$, $\beta_k:=Lt_k(1+\nu)$, and $\gamma_k:=0$ for all $k\in\N$ to derive that $\nabla f(x^k)\rightarrow0.$ Observe first that  condition \eqref{a} of this lemma is satisfied due to the the estimates
\begin{align*}
\alpha_{k+1}-\alpha_k&=\norm{\nabla f(x^{k+1})}-\norm{\nabla f(x^{k})}\\
&\le \norm{\nabla f(x^{k+1})-\nabla f(x^k)}\le L\norm{x^{k+1}-x^k}\\
&=Lt_k\norm{g^k}\le Lt_k (1+\nu)\norm{\nabla f(x^k)}=\beta_k\alpha_k.
\end{align*}
The conditions in \eqref{b} of the lemma are satisfied since $\set{t_k}$ is bounded, $\sum_{k=1}^\infty t_k =\infty$ by \eqref{stepsize IGDr}, $\gamma_k=0$, and
\begin{align*}
\sum_{k=1}^\infty \beta_k\alpha_k^2=L(1+\nu)\sum_{k=1}^\infty t_k\norm{\nabla f(x^k)}^2<\infty,
\end{align*}
where the inequality follows from \eqref{seri tk fxk con}. Thus applying Lemma~\ref{three sequences lemma} gives us $\nabla f(x^k)\rightarrow0$ as $k\to\infty$. 

To prove (iv), we verify the assumptions of Proposition~\ref{general convergence under KL} for the sequences generated by Algorithm~\ref{IGDr}. It follows from \eqref{descent esti 1} and $x^{k+1}=x^k-t_kg^k$ that 
\begin{align}
 f(x^{k+1})&\le f(x^k)-\frac{\delta t_k}{2(1+\nu)}\norm{\nabla f(x^k)}\cdot\norm{g^k}\nonumber\\
 &=f(x^k)-\frac{\delta }{2(1+\nu)}\norm{\nabla f(x^k)}\cdot\norm{x^{k+1}-x^k},
\end{align}
which justify (H1) with $\sigma=\frac{\delta}{2(1+\nu)}$. Regarding condition (H2), assume that $f(x^{k+1})=f(x^k)$ and get by \eqref{descent esti 1} that $\nabla f(x^k)=0$, which implies by $\norm{g^k-\nabla f(x^k)}\le \nu\norm{\nabla f(x^k)}$ that $g^k=0.$ Combining this with $x^{k+1}=x^k-t_kg^k$ gives us $x^{k+1}=x^k,$ which verifies (H2). Therefore, Proposition~\ref{general convergence under KL} tells us that $\set{x^k}$ is convergent.

Let us now verify the final assertion (v) of the theorem. It is nothing to prove if $\set{x^k}$ stops at a stationary point after a finite number of iterations. Thus we assume that $\nabla f(x^k)\ne 0$ for all $k\in\N.$ The assumptions in (v) give us $\bar t>0$ and $N\in\N$ such that $t_k\ge \bar t$ for all $k\ge N.$ Let us check that the assumptions of Proposition~\ref{general rate} hold for the sequences generated by Algorithm~\ref{IGDr} with $\tau_k:=t_k$ and $d^k:=-g^k$ for all $k\in\N$. The iterative procedure $x^{k+1}=x^k-t_kg^k$ can be rewritten as $x^{k+1}=x^k+t_kd^k$. Using the first condition in \eqref{two inequality} and taking into account that $\nabla f(x^k)\ne 0$ for all $k\in\N,$ we get that $g^k\ne 0$ for all $k\in\N.$ Combining this with $x^{k+1}=x^k-t_kg^k$ and $t_k\ge \bar t$ for all $k\ge N,$ tells us that $x^{k+1}\ne x^k$ for all $k\ge N.$ It follows from \eqref{descent esti 1} and \eqref{two inequality} that
\begin{align}
 f(x^{k+1})&\le f(x^k)-\frac{\delta t_k}{2(1+\nu)^2}\norm{g^k}^2\label{descent esti}.
\end{align}
This estimate together with the second inequality in \eqref{two inequality} verifies \eqref{two conditions} with $\beta=\frac{\delta}{2(1+\nu)^2},\alpha=\frac{1}{1-\nu}$. As all the assumptions are verified, Proposition~\ref{general rate} gives us the assertions: 
\begin{itemize}
\item If $q\in(0,1/2]$, then the sequence $\set{x^k}$ converges linearly to $\bar x$. 
\item If $q\in(1/2,1)$, then we have the estimate
\begin{align*}
\norm{x^k-\bar x}=\mathcal{O}\brac{ k^{-\frac{1-q}{2q-1}}}.
\end{align*}
\end{itemize}
The convergence rates of $\set{f(x^k)}$ and $\set{\norm{\nabla f(x^k)}}$ follow now from Proposition~\ref{convergence rate deduce}, and thus we are done.
\end{proof}

\subsection{Proof of Theorem~\ref{theorem USAM EG}} \label{proof USAM EG}
\begin{proof} Let $\set{x^k}$ be the sequence generated by Algorithm~\ref{USAM}. Defining $g^k:=\nabla f(x^k+\rho_k\nabla f(x^k))$ and utilizing  $\rho_k \le \frac{\nu}{L}$, we obtain
\begin{align*}
\norm{g^k-\nabla f(x^k)}&=\norm{\nabla f(x^k+\rho_k\nabla f(x^k))-\nabla f(x^k)}\\
&\le L\norm{\rho_k \nabla f(x^k)}\le \nu \norm{\nabla f(x^k)},
\end{align*}
which verifies the inexact condition in Step~2 of Algorithm~\ref{IGDr}. Therefore, all the convergence properties in Theorem~\ref{theo IGDr} hold for Algorithm~\ref{USAM}. The proof for the convergence properties of Algorithm~\ref{EG} can be conducted similarly.
\end{proof}
\subsection{Proof of Corollary~\ref{coro general}} \label{proof coro general}
\begin{proof}
Considering Algorithm~\ref{general} and defining $g^k = \nabla f\left(x^k + \rho_k \frac{d^k}{\norm{d^k}}\right)$, we deduce that
\begin{align*}
\norm{g^k-\nabla f(x^k)}\le L\norm{x^k+\rho_k\frac{d^k}{\norm{d^k}}-x^k}=L\rho_k.
\end{align*}
Therefore, Algorithm~\ref{general} is a specialization of Algorithm~\ref{IGD} with $\varepsilon_k = L\rho_k.$ Combining this with \eqref{parameter SAM + VASSO} also gives us \eqref{parameter general}, thereby verifying all the assumptions in Theorem~\ref{convergence IGD}. Consequently, all the convergence properties outlined in Theorem~\ref{convergence IGD} hold for Algorithm~\ref{general}.
\end{proof}
\section{Efficient normalized variants of SAM}\label{versions}
In this section, we list several efficient normalized variants of SAM from \citep{foret21,liu22,li23vasso,li2024friendly} that are special cases of Algorithm~\ref{general}. As a consequence, all the convergence properties in Theorem~\ref{convergence IGD} are satisfied for these methods.
\begin{algorithm}[H]\caption{\citep{foret21} Sharpness-Aware Minimization (SAM)}\label{SAM}
\setcounter{Step}{-1}
 \begin{Step}\rm
Choose $x^1\in\R^n,$ $\set{\rho_k}\subset [0,\infty)$, and $\set{t_k}\subset[0,\infty).$ For $k=1,2,\ldots,$ do the following:
\end{Step}
\begin{Step}\rm 
Set $x^{k+1}=x^k-t_k \nabla f\brac{x^k+\rho_k \frac{\nabla f(x^k)}{\norm{\nabla f(x^k)}}}$.
\end{Step}   
\end{algorithm}
\begin{algorithm}[H]\caption{\citep{liu22} Random Sharpness-Aware Minimization (RSAM)}\label{RSAM}
\setcounter{Step}{-1}
 \begin{Step}\rm
Choose $x^1\in\R^n,$  $\set{\rho_k}\subset [0,\infty)$, and $\set{t_k}\subset[0,\infty).$ For $k=1,2,\ldots,$ do the following:
\end{Step}
\begin{Step}
Construct a random vector $\Delta^k\in\R^n$ and set $g^k=\nabla f(x^k+\Delta^k).$ 
\end{Step}
\begin{Step}\rm 
Set $x^{k+1}=x^k-t_k \nabla f\brac{x^k+\rho_k \frac{\Delta^k+\lambda g^k}{\norm{\Delta^k+\lambda g^k}}}$.
\end{Step} 
\end{algorithm}

\begin{algorithm}[H]
\caption{\citep{li23vasso} Variance suppressed sharpness aware optimization (VaSSO) }\label{VaSSO}
   \setcounter{Step}{-1}
 \begin{Step}\rm
Choose $x^1\in\R^n$, $d^1\in\R^n$, $\set{\rho_k}\subset [0,\infty)$, $\set{t_k}\subset[0,\infty).$ For $k\ge 1,$ do the following:
\end{Step}
\begin{Step}\rm 
Set $d^{k}=(1-\theta)d^{k-1}+\theta \nabla f(x^k).$
\end{Step}   
\begin{Step}
Set $x^{k+1}=x^k-t_k \nabla f\brac{x^k+\rho_k \frac{d^k}{\norm{d^k}}}.$
\end{Step} 
\end{algorithm}

\begin{algorithm}[H]
\caption{\citep{li2024friendly} Friendly Sharpness-Aware Minimization (F-SAM) }\label{F-SAM}
   \setcounter{Step}{-1}
     \begin{Step}\rm
        Choose $x^1\in\R^n$, $d^1\in\R^n$, $m^1\in\R^n$, $\sigma\in\R$, $\set{\rho_k}\subset [0,\infty)$, $\set{t_k}\subset[0,\infty).$ For $k\ge 1$:
    \end{Step}
    \begin{Step}\rm 
        Set $m^{k}=(1-\theta)m^{k-1}+\theta \nabla f(x^k).$
    \end{Step}   
    \begin{Step}\rm 
        Set $d^{k}= \nabla f(x^k) - \sigma m^{k}.$
    \end{Step}   
    \begin{Step}
        Set $x^{k+1}=x^k-t_k \nabla f\brac{x^k+\rho_k \frac{d^k}{\norm{d^k}}}.$
    \end{Step} 
\end{algorithm}

\begin{remark}\rm\label{normalized variants remark}

It is clear that Algorithms~\ref{SAM}-\ref{F-SAM} are specializations of Algorithm~\ref{general} with $d^k=\nabla f(x^k)$ in Algorithm~\ref{SAM}, $d^k=\Delta^k+\lambda g^k$ in Algorithm~\ref{RSAM}, and $d^k$ constructed inductively for Algorithm~\ref{VaSSO} and Algorithm~\ref{F-SAM}.
\end{remark}

\section{Numerical experiments on SAM constant and SAM almost constant}\label{numerical support}
In this section, we present numerical results to support our claim in Remark~\ref{rmk no harm} that SAM with an almost constant perturbation radius $\rho_k=\frac{C}{k^p}$ for $p$ close to $0$, e.g., $p=0.001$, generates similar results to SAM with a constant perturbation radius $\rho=C$. To do so, we consider the function $f(x)=\sum_{i=1}^n\log(1+(Ax-b)_i^2)$, where $A$ is an $n\times n$ matrix, and $b$ is a vector in $\mathbb{R}^n$. In the experiment, we construct $A$ and $b$ randomly with $n\in{2,20,50,100}$. The methods considered in the experiment are GD with a diminishing step size, SAM with a diminishing step size and a constant perturbation radius of $0.1$, and lastly, SAM with a diminishing step size and a variable radius $\rho_k=\frac{C}{k^p}$, for $p\in{1,0.1,0.001}$. We refer to the case $p=0.001$ as the "almost constant" case, as $\rho_k=\frac{C}{k^p}$ is numerically similar to $C$ when we consider a small number of iterations. The diminishing step size is chosen as $t_k=(0.1/n)/k$ at the $k^{\rm th}$ iteration, where $n$ is the dimension of the problem. To make the plots clearer, we choose the initial point $x^1$ near the solution, which is $x^1=x^\infty+(0.1/n^2)\mathbf{1}_n$, where $x^\infty$ is a solution of $Ax=b$, and $\mathbf{1}_n$ is the all-ones vector in $\mathbb{R}^n$. All the algorithms are executed for $100n$ iterations. The results presented in Figure~\ref{fig:SAM nonconvex} show that SAM with a constant perturbation and SAM with an almost constant perturbation have the same behavior regardless of the dimension of the problem. This is simply because $\frac{C}{k^{0.001}}$ is almost the same as $C$. This also tells us that the convergence rate of these two versions of SAM is similar. Since SAM with a constant perturbation radius is always preferable in practice \citep{foret21,dai23}, this highlights the practicality of our development.

\begin{figure}[H]
    \centering
    \subfigure[$n=2$]{
        \includegraphics[width=1\linewidth]{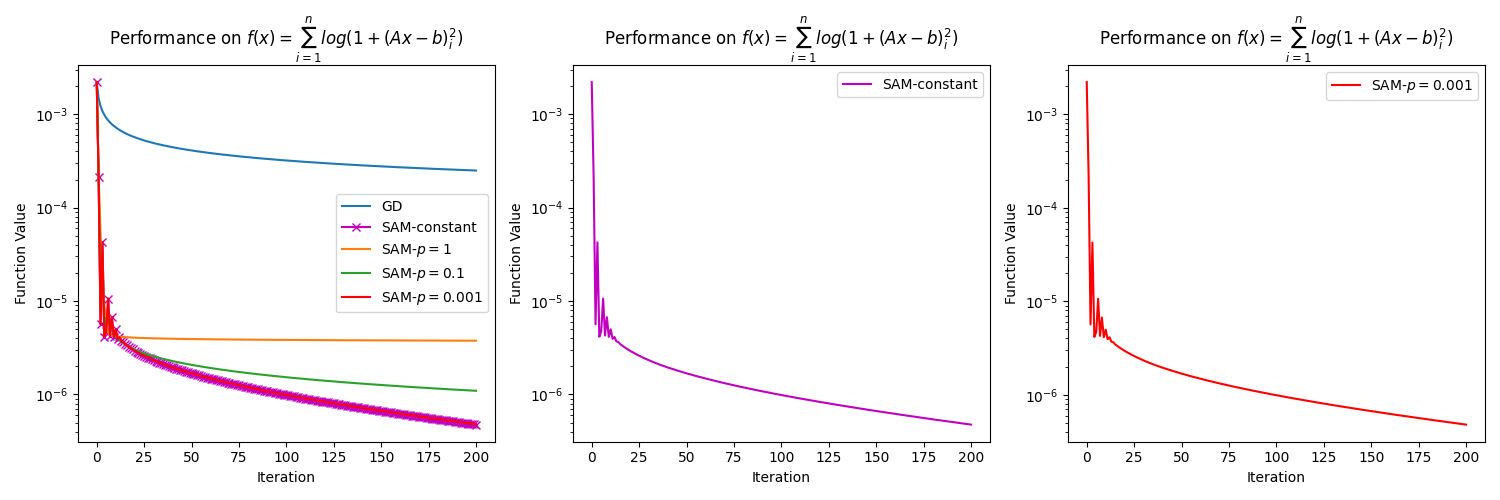}
        \label{fig:n_2}
    }
    \subfigure[$n=20$]{
        \includegraphics[width=1\linewidth]{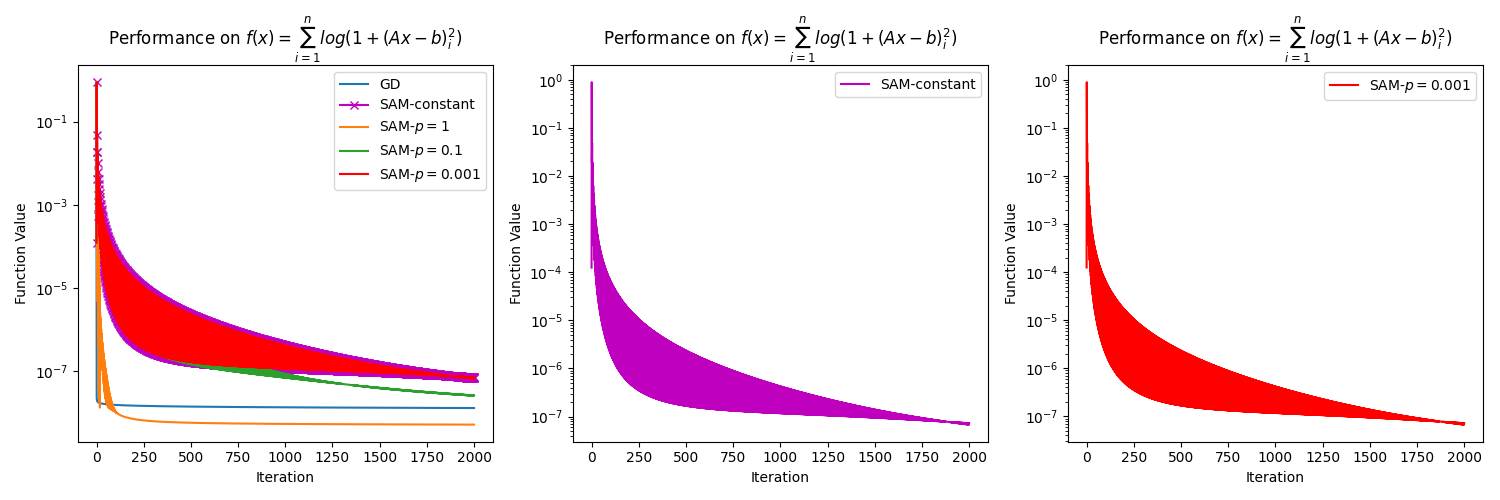}
        \label{fig:n_20}
    }
    \\
    \subfigure[$n=50$]{
        \includegraphics[width=1\linewidth]{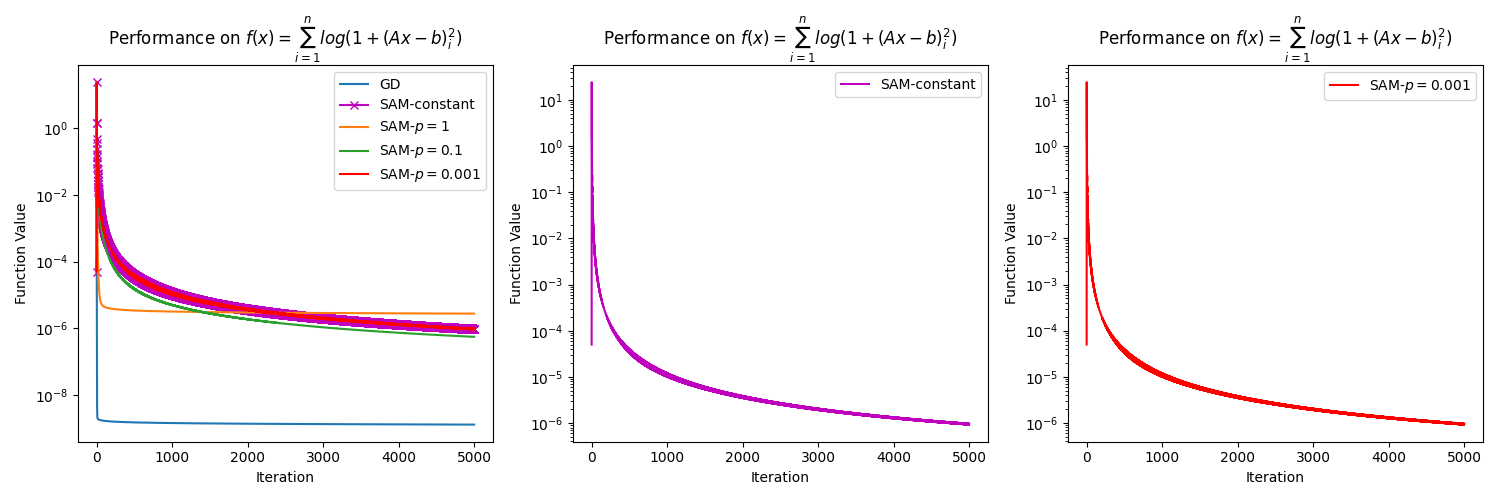}
        \label{fig:n_50}
    }
    \subfigure[$n=100$]{
        \includegraphics[width=1\linewidth]{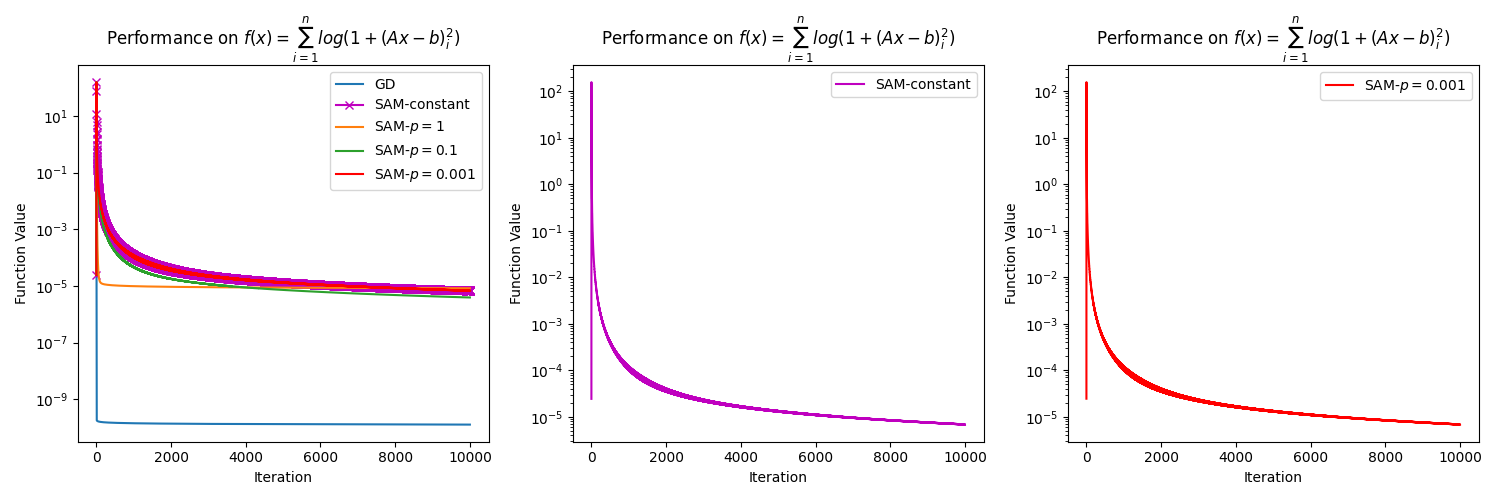}
        \label{fig:n_100}
    }
    \caption{SAM with constant perturbation and SAM almost constant perturbation}
    \label{fig:SAM nonconvex}
\end{figure}

\section{Numerical experiments on SAM with SGD without momentum as base optimizer}\label{nume addi}

\textbf{CIFAR-10, CIFAR-100, and Tiny ImageNet}. The training configurations for these datasets follow a similar structure to Section \ref{sec numerical}, excluding momentum, which we set to zero. The results in Table \ref{table: experiment without momentum} report test accuracy on CIFAR-10 and CIFAR-100. Table \ref{table: experiment tiny} shows the performance of SAM on momentum and without momentum settings. Each experiment is run once, and the highest accuracy for each column is highlighted in bold.

\begin{table}[H]\small
\centering
\caption{Additional Numerical Results on CIFAR-10, CIFAR-100 for SAM without momentum}
\label{table: experiment without momentum}

\begin{tblr}{
  column{even} = {r},
  column{3} = {r},
  column{5} = {r},
  column{7} = {r},
  cell{1}{2} = {c=3}{c},
  cell{1}{5} = {c=3}{c},
  hline{1-3,7} = {-}{},
}
             & CIFAR-10&         &         & CIFAR-100                       &         &         \\
Model & ResNet18& ResNet34& WRN28-10& ResNet18& ResNet34& WRN28-10\\
Constant     & 93.64 & 94.26  	& 93.04 & 72.07 & 72.57 & 71.11 \\
Diminish 1   & 88.87 & 89.79 & 86.81 & 65.99 & 67.04 & 51.31 \\
Diminish 2   & \textbf{94.56} & \textbf{94.44}& \textbf{93.66} & \textbf{74.24} & \textbf{74.95} & \textbf{74.23} \\
Diminish 3   & 90.84 & 91.23 & 88.70 & 69.69 & 70.54 & 60.64
\end{tblr}
\end{table}

\begin{table}[H]\small
\centering
\caption{Additional Numerical Results on TinyImagenet for SAM with and without momentum}
\label{table: experiment tiny}

\begin{tblr}{
  column{even} = {r},
  column{3} = {r},
  column{5} = {r},
  column{7} = {r},
  cell{1}{2} = {c=3}{c},
  cell{1}{5} = {c=3}{c},
  hline{1-3,7} = {-}{},
}
Tiny ImageNet & Momentum&         &         & Without Momentum                       &         &         \\
Model & ResNet18& ResNet34& WRN28-10& ResNet18& ResNet34\\
Constant   & 48.58 & 48.34 & 53.34 & 54.90 & 57.36  \\
Diminish 1 & 50.36 & 51.24 & 58.37 & 55.82 & 55.96   \\
Diminish 2 & 48.70 & 49.06 & 52.74 & 57.30 & \textbf{60.00} \\
Diminish 3 & \textbf{51.46} & \textbf{53.98} & \textbf{58.68} & \textbf{57.86} & 57.82 
\end{tblr}
\end{table}

\section{Additional Remarks}
\begin{remark} \rm \label{rmk satisfaction} Assumption~\ref{assu desi} is satisfied with constant $C=1$ for $\varphi(t)=Mt^{1-q}$ with $M>0$ and $q\in [0,1)$.
Indeed, taking any $x,y>0$ with $x+y<\eta$, we deduce that $(x+y)^q\le x^q+y^q,$ and hence
\begin{align*}
 [\varphi'(x+y)]^{-1}&=\frac{1}{M(1-q)}(x+y)^q\le \frac{1}{M(1-q)}(x^q+y^q)
= (\varphi'(x))^{-1}+(\varphi'(y))^{-1}.
\end{align*}
\end{remark}

\begin{remark}\label{not converge to 0}
Construct an example to demonstrate that the conditions in \eqref{parameter SAM + VASSO} do not require that ${\rho_k}$ converges to $0$. Let $L>0$ be a Lipschitz constant of $\nabla f$, let $C$ be a positive constant such that $C<2/L$, let $P\subset\N$ be the set of all perfect squares, let $t_k=\frac{1}{k}$ for all $k\in\N,p>0$, and let $\set{\rho_k}$ be constructed as follows:
\begin{align*}
\rho_k=\begin{cases}C,&k\in P,\\
  \frac{C}{k^p},&k\notin P, 
\end{cases}\;\text{ which yields }\;
{t_k}{\rho_k}=\begin{cases}\frac{C}{k},&k\in P,\\
  \frac{C}{k^{p+1}},&k\notin P.  
\end{cases}
\end{align*}
It is clear from the construction of $\set{\rho_k}$ that $\limsup_{k\rightarrow\infty}\rho_k=C>0$, which implies that $\set{\rho_k}$ does not convergence to $0.$ We also immediately deduce that $ \sum_{k=1}^\infty t_k=\infty,\;t_k\downarrow0,$ and $\limsup \rho_k=C<\frac{2}{L},$ which verifies the first three conditions in \eqref{parameter SAM + VASSO}. The last condition in \eqref{parameter SAM + VASSO} follows from the estimates
\begin{align*}
\sum_{k=1}^\infty {t_k}{\rho_k}&=\sum_{k\in P} {t_k}{\rho_k}+\sum_{k\notin P} {t_k}{\rho_k}\le \sum_{k\in P} \frac{C}{k}+\sum_{k\in \N}\frac{C}{k^{p+1}}\\
&=\sum_{k\in \N} \frac{C}{k^2}+\sum_{k\in \N}\frac{C}{k^{p+1}}<\infty.
\end{align*}
\end{remark}

\begin{remark}[on Assumption~\eqref{desing condi}]\rm\label{rmk desing condi}
Supposing that $\varphi(t)=Mt^{1-q}$ with $M>0,q\in(0,1)$ and letting $C=1/(M(1-q))$, we get that $(\varphi'(t))^{-1}=Ct^q\;\text{ for }\;t>0$ is an increasing function.
If $t_k:=\frac{1}{k}$ and $\varepsilon_k:=\frac{1}{k^p}$ with $p>0$, we have
\begin{align*}
\sum_{i=k}^\infty t_i\varepsilon_i= \sum_{i=k}^\infty \frac{1}{i^{1+p}}\le \int_{k}^\infty \frac{1}{x^{1+p}}dx=-\frac{1}{px^p}\big|_{k}^\infty=\frac{1}{pk^p},
\end{align*}
which yields the relationships
\begin{align*}
\brac{\varphi'\brac{\sum_{i=k}^\infty t_i\varepsilon_i}}^{-1}\le \brac{\varphi'\brac{\frac{1}{p(k+1)^p}}}^{-1}=\frac{C}{p^qk^{pq}}.
\end{align*}
Therefore, we arrive at the claimed conditions
\begin{align*}
\sum_{k=1}^\infty t_k\brac{\varphi'\brac{\sum_{i=k}^\infty t_i\varepsilon_i}}^{-1}\le \sum_{k=1}^\infty \frac{1}{k}\frac{C}{p^qk^{pq}}= \sum_{k=1}^\infty \frac{C}{p^qk^{1+pq}}<\infty.
\end{align*}
\end{remark}

\begin{remark}\rm \label{rmk comparison USAM} Let us finally compare the results presented in Theorem~\ref{theo IGDr} with that in \citet{maksym22}. All the convergence properties in \citet{maksym22} are considered for the class of $\C^{1,L}$ functions, which is more narrow than the class of $L$-descent functions examined in Theorem~\ref{theo IGDr}(i). Under the convexity of the objective function, the convergence of the sequences of the function values at \textit{averages of iteration} is established in \citet[Theorem~11]{maksym22}, which does not yield the convergence of either the function values, or the iterates, or the corresponding gradients. In the nonconvex case, we derive the stationarity of accumulation points, the convergence of the function value sequence, and the convergence of the gradient sequence in Theorem~\ref{theo IGDr}. Under the strong convexity of the objective function, the linear convergence of the sequence of iterate values is established \citet[Theorem~11]{maksym22}. On the other hand, our Theorem~\ref{theo IGDr} derives the convergence rates for the sequence of iterates, sequence of function values, and sequence of gradient under the KL property only, which covers many classes of nonconvex functions. Our convergence results address variable stepsizes and bounded radii, which also cover the case of constant stepsize and constant radii considered in \citet{maksym22}.
\end{remark}

\newpage
\section*{NeurIPS Paper Checklist}

\begin{enumerate}

\item {\bf Claims}
\item[] Question: Do the main claims made in the abstract and introduction accurately reflect the paper's contributions and scope?
\item[] Answer: \answerYes
\item[] Justification: The claims made in the abstract and introduction about convergence properties of SAM and its variants are presented in Section~\ref{sec SAM normal} and Section~\ref{sec uSAM}.

\item {\bf Limitations}
\item[] Question: Does the paper discuss the limitations of the work performed by the authors?
\item[] Answer: \answerYes
\item[] Justification: See Section~\ref{sec discuss}.

\item {\bf Theory Assumptions and Proofs}
\item[] Question: For each theoretical result, does the paper provide the full set of assumptions and a complete (and correct) proof?
\item[] Answer: \answerYes
\item[] Justification: All the assumptions are given in the main text while the full proofs are provided in the appendices.

\item {\bf Experimental Result Reproducibility}
\item[] Question: Does the paper fully disclose all the information needed to reproduce the main experimental results of the paper to the extent that it affects the main claims and/or conclusions of the paper (regardless of whether the code and data are provided or not)?
\item[] Answer: \answerYes{} 
\item[] Justification: See Section \ref{sec numerical}

\item {\bf Open access to data and code}
\item[] Question: Does the paper provide open access to the data and code, with sufficient instructions to faithfully reproduce the main experimental results, as described in supplemental material?
\item[] Answer: \answerYes{}.
\item[] Justification: See the supplementary material. 

\item {\bf Experimental Setting/Details}
\item[] Question: Does the paper specify all the training and test details (e.g., data splits, hyperparameters, how they were chosen, type of optimizer, etc.) necessary to understand the results?
\item[] Answer: \answerYes{} 
\item[] Justification: See Section \ref{sec numerical}

\item {\bf Experiment Statistical Significance}
\item[] Question: Does the paper report error bars suitably and correctly defined or other appropriate information about the statistical significance of the experiments?
\item[] Answer: \answerYes{} 
\item[] Justification: See Section \ref{sec numerical}

\item {\bf Experiments Compute Resources}
\item[] Question: For each experiment, does the paper provide sufficient information on the computer resources (type of compute workers, memory, time of execution) needed to reproduce the experiments?
\item[] Answer: \answerYes{} 
\item[] Justification: We mentioned an RTX 3090 computer worker.

\item {\bf Code Of Ethics}
\item[] Question: Does the research conducted in the paper conform, in every respect, with the NeurIPS Code of Ethics \url{https://neurips.cc/public/EthicsGuidelines}?
\item[] Answer: \answerYes{} 
\item[] Justification: Our paper conform with the NeurIPS Code of Ethics.

\item {\bf Broader Impacts}
\item[] Question: Does the paper discuss both potential positive societal impacts and negative societal impacts of the work performed?
\item[] Answer: \answerNA{} 
\item[] Justification: This paper presents work whose goal is to advance the field of Machine Learning. There are many potential societal consequences of our work, none of which are specifically highlighted here.

\item {\bf Safeguards}
\item[] Question: Does the paper describe safeguards that have been put in place for responsible release of data or models that have a high risk for misuse (e.g., pretrained language models, image generators, or scraped datasets)?
\item[] Answer: \answerNA{}  
\item[] Justification: Our paper is a theoretical study.

\item {\bf Licenses for existing assets}
\item[] Question: Are the creators or original owners of assets (e.g., code, data, models), used in the paper, properly credited and are the license and terms of use explicitly mentioned and properly respected?
\item[] Answer: \answerYes{} 
\item[] Justification: See Section~\ref{sec numerical}.

\item {\bf New Assets}
\item[] Question: Are new assets introduced in the paper well documented and is the documentation provided alongside the assets?
\item[] Answer: \answerNA{} 
\item[] Justification: Our paper is a theoretical study.

\item {\bf Crowdsourcing and Research with Human Subjects}
\item[] Question: For crowdsourcing experiments and research with human subjects, does the paper include the full text of instructions given to participants and screenshots, if applicable, as well as details about compensation (if any)? 
\item[] Answer: \answerNA{} 
\item[] Justification: 

\item {\bf Institutional Review Board (IRB) Approvals or Equivalent for Research with Human Subjects}
\item[] Question: Does the paper describe potential risks incurred by study participants, whether such risks were disclosed to the subjects, and whether Institutional Review Board (IRB) approvals (or an equivalent approval/review based on the requirements of your country or institution) were obtained?
\item[] Answer: \answerNA{} 
\item[] Justification: 

\end{enumerate}

\end{document}